\documentclass[12pt,a4paper]{amsart}
\usepackage[numbers]{natbib}
\usepackage{pst-node}
\usepackage{tikz-cd} 

\usepackage{enumitem,kantlipsum}
\usepackage[english]{babel}

\usepackage[T1]{fontenc}
\usepackage{enumitem}
\usepackage[utf8]{inputenc}

\usepackage{ifxetex}
\usepackage{pb-diagram}
\usepackage{amsfonts}
\usepackage{amsmath, amsthm, amscd, amsfonts, xpatch}
\usepackage{caption}
\usepackage{graphicx,float}
\usepackage{tikz}
\usepackage{setspace}
\usepackage{lineno,hyperref}
\usepackage{graphicx,float}
\usepackage[english]{babel}
\usepackage{bussproofs}
\usepackage{csquotes}
\usepackage{dirtytalk}
\usepackage{mathtools}
\usepackage{mathrsfs}
\usepackage{latexsym}
\usepackage{enumitem}
\usepackage{amsthm}
\usepackage{amssymb}
\usepackage{amsthm}

\DeclareMathAlphabet{\mathpzc}{OT1}{pzc}{m}{it}
\usepackage{afterpage}

\usepackage{mathrsfs}

\usepackage[capitalise]{cleveref}

\input xy
\xyoption{all}

\newcommand{\ORD}{{\rm ORD}}
\newtheorem{theorem}{Theorem}[section]
\newtheorem{lemma}[theorem]{Lemma}

\newtheorem{proposition}[theorem]{Proposition}
\newtheorem{corollary}[theorem]{Corollary}

\newtheorem{definition}[theorem]{Definition}

\newtheorem{notation}[theorem]{Notation}
\newtheorem{claim}[theorem]{Claim}
\numberwithin{equation}{section}
\newtheorem{convention}[theorem]{Convention}

\newcommand{\cof}{{\rm cof}}
\newcommand{\dom}{{\rm dom}}

\newcommand{\M}{\mathcal{M}}

\newcommand{\rest}{\! \upharpoonright \!}

\newcommand\axiom{\mathrm}

\newcommand\ZFC{\axiom{ZFC}}

\newcommand{\forces}{\Vdash}

\theoremstyle{remark}
\newtheorem{remark}[theorem]{Remark}
\makeatletter
\let\qed@empty\openbox 
\def\@begintheorem#1#2[#3]{%
  \deferred@thm@head{%
    \the\thm@headfont\thm@indent
    \@ifempty{#1}
      {\let\thmname\@gobble}
      {\let\thmname\@iden}%
    \@ifempty{#2}
      {\let\thmnumber\@gobble\global\let\qed@current\qed@empty}
      {\let\thmnumber\@iden\xdef\qed@current{#2}}%
    \@ifempty{#3}
      {\let\thmnote\@gobble}
      {\let\thmnote\@iden}%
    \thm@swap\swappedhead
    \thmhead{#1}{#2}{#3}%
    \the\thm@headpunct\thmheadnl\hskip\thm@headsep
  }\ignorespaces
}
\renewcommand{\qedsymbol}{%
  \ifx\qed@thiscurrent\qed@empty
    \qed@empty
  \else
    \fbox{\scriptsize\qed@thiscurrent}%
  \fi
}
\renewcommand{\proofname}{%
  Proof%
  \ifx\qed@thiscurrent\qed@empty
  \else
    \ of \qed@thiscurrent
  \fi
}
\xpretocmd{\proof}{\let\qed@thiscurrent\qed@current}{}{}
\newenvironment{proof*}[1]
  {\def\qed@thiscurrent{\ref{#1}}\proof}
  {\endproof}
\makeatother
\begin{document}

\nocite{*}

\title{Guessing models and the approachability ideal} 
\address{Institut de Math\'ematiques de Jussieu - Paris Rive Gauche,
Universit\'e Paris Diderot,  8 Place Aur\'elie Nemours, 75205 Paris Cedex 13, France}

\author{Rahman Mohammadpour}
\email{rahmanmohammadpour@gmail.com}
\urladdr{https://sites.google.com/site/rahmanmohammadpour}

\author[Boban Veli\v{c}kovi\'{c}]{Boban Veli\v{c}kovi\'{c}}

\email{boban@math.univ-paris-diderot.fr}
\urladdr{http://www.logique.jussieu.fr/~ boban}

\keywords{ approachability ideal, guessing models, supercompact cardinals}
\subjclass[2010]{03E35, 03E55, 03E05}

\begin{abstract} 
Starting with two supercompact cardinals we produce a generic extension of the universe in which a principle
that we call ${\rm GM}^+(\omega_3,\omega_1)$ holds. 
This principle implies ${\rm ISP}(\omega_2)$ and ${\rm ISP}(\omega_3)$, and hence 
the tree property at $\omega_2$ and $\omega_3$, the Singular Cardinal Hypothesis, and 
the failure of the weak square principle $\square(\omega_2,\lambda)$, for all regular $\lambda \geq \omega_2$. 
In addition, it implies that the restriction of the approachability
ideal  $I[\omega_2]$ to the set of  ordinals of cofinality $\omega_1$ is the non stationary ideal on this set.
The consistency of this last statement was previously shown by Mitchell. 
\end{abstract}

\maketitle

\section*{Introduction}

In \cite{W2012} C. Wei\ss~ formulated some combinatorial principles that capture the essence of some large cardinal properties, but can hold at small cardinals. 
These principles usually have two parameters, a regular uncountable cardinal $\kappa$ and a cardinal $\lambda \geq \kappa$. 
Among them there are, in increasing strength, the principles ${\rm TP}(\kappa,\lambda)$, ${\rm ITP}(\kappa, \lambda)$, and ${\rm ISP}(\kappa,\lambda)$.
We will write ${\rm P}(\kappa)$, if the property ${\rm P}(\kappa,\lambda)$ holds, for all $\lambda \geq \kappa$.
The study of these principles was continued by M. Viale and C. Wei\ss~ in \cite{VW2011}. 
Using them they obtained  a striking result  saying that any standard forcing construction of a model of the Proper Forcing Axiom $({\rm PFA})$ 
requires at least a strongly compact cardinal. 
One important concept that emerged from this work is that  of a guessing model. 
These models have generated considerable interest and have a number of interesting applications, see for instance  
\cite{V2012}, \cite{CK}, \cite{CK2017},  and \cite{T2016}. 

Given the interest of these principles, it is natural to ask if they can hold simultaneously at several successive regular cardinals. 
 In this direction, L. Fontanella \cite{Fontanella2012} extended the previous work of U. Abraham \cite{Abraham1983} 
 to obtain, modulo two supercompact cardinals, a model  of $\ZFC$ in which ${\rm ITP}(\omega_2)$ and ${\rm ITP}(\omega_3)$ hold simultaneously. 
Now, it was shown in \cite{W2012} that ${\rm ISP}(\omega_2)$ is strictly stronger that ${\rm ITP}(\omega_2)$. In fact, in the model constructed
by B. K\"onig in \cite{Konig2007}, the principle ${\rm ITP}(\omega_2)$ holds, but ${\rm ISP}(\omega_2)$ fails. 
One can then ask if ${\rm ISP}(\omega_2)$ and ${\rm ISP}(\omega_3)$ can hold simultaneously. 
Let us point out that in \cite{T2016} Trang showed the consistency of ${\rm ISP}(\omega_3)$. 
However, in his model $\rm CH$ holds, and therefore the principle ${\rm ISP}(\omega_2)$ fails. 

One concept closely related to the above principles is that of the approachability property on a regular uncountable cardinal $\lambda$
and the associated ideal $I[\lambda]$.
These notions were introduced by Shelah implicitly in \cite{Shelah1979}, and studied by him  extensively over the past 40 years. 
For instance, in \cite{Shelah1991} he showed that if $\lambda$ is a regular cardinal then $S_{\lambda^+}^{<\lambda} \in I[\lambda^+]$,
and in \cite{Shelah1993} he showed that if $\kappa$ is regular and $\kappa^+ < \lambda$ then $I[\lambda]$ contains
a stationary subset of $S_{\lambda}^\kappa$. 
Shelah then asked in \cite{Shelah1991} if  it is consistent to have a regular $\lambda$ such that $I[\lambda^+]\rest S_{\lambda^+}^\lambda$
 is the non stationary ideal on $S_{\lambda^+}^\lambda$. This major question was finally answered by W. Mitchell \cite{MI2009}.
He started with a cardinal $\kappa$ that is $\kappa^+$-Mahlo, and built an involved forcing construction 
yielding a model in which $I[\omega_2]\rest S_{\omega_2}^{\omega_1}$ is the non stationary ideal on $S_{\omega_2}^{\omega_1}$.
One feature of this construction is that it uses $\square_\kappa$ in the ground model, and so  $\omega_3\in I[\omega_3]$ in the extension.  
It is therefore unclear if Mitchell's method can be adapted to obtain a model in which both $I[\omega_2]\rest S_{\omega_2}^{\omega_1}$
and $I[\omega_3]\rest S_{\omega_3}^{\omega_2}$  contain only non stationary sets. 
The connection with the principles introduced by Wei\ss \, is the following. 
If $\kappa$ is a regular uncountable cardinal then ${\rm ISP}(\kappa^+)$ implies that there is a stationary subset of $S_{\kappa^+}^\kappa$
that is not in $I[\kappa^+]$, but it does not imply Mitchell's result. 
The main purpose of this paper is to formulate and show the relative consistency of a principle that we call ${\rm GM}^+(\omega_3,\omega_1)$.
This statement implies  ${\rm ISP}(\omega_2)$ and ${\rm ISP}(\omega_3)$ and hence the tree property 
at $\omega_2$ and $\omega_3$. It also implies Mitchell's result, namely that 
$I[\omega_2]\rest S_{\omega_2}^{\omega_1}= {\rm NS}_{\omega_2}\rest S_{\omega_2}^{\omega_1}$. 
Starting with a model with two supercompact cardinals $\kappa < \lambda$ we produce a generic
extension in which $\kappa=\omega_2$, $\lambda = \omega_3$, and ${\rm GM}^+(\omega_3,\omega_1)$ holds. 
In fact for Mitchell's result we do not need the full strength of the principle ${\rm GM}^+(\omega_3,\omega_1)$, 
a weaker principle ${\rm FS}(\omega_2,\omega_1)$ suffices. In order to obtain ${\rm FS}(\omega_2,\omega_1)$ it is enough
to assume that $\lambda$ is just inaccessible.

The origin of this paper is as follows. Inspired by Mitchell's breakthrough, I. Neeman \cite{NE2014} 
introduced a method for iterating proper forcing notions using finite chains of elementary submodels as side conditions. 
This allowed him to give a new proof of the consistency of PFA using this type of iteration. 
More importantly, this opened a possibility of iterating forcing while preserving two successive cardinals
and potentially getting strong forcing axioms at $\omega_2$ and higher cardinals. 
The second author then extended Neeman style iteration to more general classes of forcing.
This lead to the notion of a {\em virtual model}. Using this type of models as side conditions allows us
not only to  generalize Neeman's iteration theory to semiproper forcing, but also to formulate and prove iteration
theorems for  large classes of forcing notions preserving two uncountable cardinals, such as $\omega_1$ and $\omega_2$. 
This theory is presented in \cite{Velisemiproper} and \cite{Veliproper}. 
In fact, our main poset is an adaptation of the pure side condition forcing from \cite{Veliproper} to two types
of models, but replacing models of size $\omega_1$ by  models having a strong closure property
that we call Magidor models. 
  
The paper is organized as follows. In \S 1 we present the preliminaries and fix some notations. 
In \S 2 we review the theory of {\rm virtual models} from \cite{Velisemiproper} and \cite{Veliproper}
and adapt it to the context of Magidor models. 
In \S 3, we introduce our main forcing notion $\mathbb P^\kappa_\lambda$ and establish some of its properties. 
Finally, in \S 4  we study guessing models in the generic extension by $\mathbb P^\kappa_\lambda$,
and show that if $\kappa$ is supercompact and $\lambda> \kappa$ is inaccessible then 
${\rm ISP}(\omega_2)$ and ${\rm FS}(\omega_2,\omega_1)$ hold in the generic extension. 
If $\lambda$ is also supercompact we show that ${\rm GM}^+(\omega_3,\omega_1)$ holds as well.

\section{Preliminaries}
  
Throughout this paper by a model $M$ we mean a set or a class such that 
$(M,\in)$ satisfies a sufficient fragment of ${\rm ZFC}$. 
For a model $M$, we let $\overline{M}$ denote its transitive collapse and we let $\pi_M$ be the collapsing map. 
For a set $X$  and an uncountable regular cardinal $\kappa$, we let $\mathcal P_\kappa(X)$ denote the set of all subsets of $X$ of size less than $\kappa$.
We say that a subset $\mathcal S$ of $\mathcal P_{\kappa}(X)$ is {\em stationary} if, for every function $F:[X]^{<\omega}\rightarrow X$, 
there exists $A\in \mathcal S$ such that $A\cap \kappa \in \kappa$ and $A$ is closed under $F$. 
For regular cardinals $\kappa < \lambda$, we let $S_\lambda^\kappa$ denote the set $\{ \alpha < \lambda : \cof(\alpha)=\kappa \}$. 
For a cardinal $\theta$ we let $H_\theta$ denote  the collection of all sets whose transitive closure has size less than $\theta$.

We now recall some relevant definitions from \cite{V2012}.  For a set or class $M$ we say that a set $x\subseteq M$ is
{\em bounded in} $M$ if there is $y\in M$ such that $x\subseteq y$. 

\begin{definition}\label{guessing-def} Let $\gamma$ be a regular cardinal. 
A set or a class $M$ is said to be $\gamma$-{\em guessing} if for any $x\subseteq M$ which is bounded in $M$,
if $x$ is $\gamma$-{\em approximated by} $M$, i.e. $x\cap a \in M$, for all $a\in M\cap {\mathcal P}_\gamma (M)$,
then $x$ is $M$-{\em guessed}, i.e. there is $g\in M$ such that $x= g\cap M$.
\end{definition} 

We say that a transitive model $R$ is a {\em powerful} if it is closed under taking subsets, i.e.
if $x\in R$ and $y\subseteq x$ then $y\in R$. 
We are mainly interested in the case $R= V_\alpha$, for some ordinal $\alpha$, or $R= H_\theta$, for some uncountable regular cardinal $\theta$.  
For a powerful model $R$, a regular cardinal $\gamma$ and a cardinal $\kappa$, we let 

\[
\mathfrak G_{\kappa,\gamma}(R) = \{ M\in {\mathcal P}_\kappa(R): M\prec R  \mbox{ and } M \mbox{ is $\gamma$-guessing} \}. 
\]

\begin{definition}\label{GM}
For a powerful model $R$, ${\rm GM}(\kappa,\gamma,R)$ is the statement that  
${\mathfrak G}_{\kappa,\gamma}(R)$ is stationary in ${\mathcal P}_{\kappa}(R)$.
${\rm GM}(\kappa,\gamma)$ is the statement that ${\rm GM}(\kappa,\gamma, H_\theta)$ holds,
 for all sufficiently large regular $\theta$. 
\end{definition}

\begin{remark} Notice that if $M$ is $\gamma$-guessing and $\gamma\leq \gamma'$ then $M$ is $\gamma'$-guessing.
Therefore  ${\rm GM}(\kappa, \gamma)$ implies ${\rm GM}(\kappa, \gamma')$.
\end{remark}

The statement ${\rm GM}(\kappa,\omega_1)$ is a reformulation of the principle ${\rm ISP}(\kappa)$ introduced
by  C. Wei\ss \, in \cite{W2010} and further studied in \cite{V2012},  \cite{VW2011} and \cite{W2012}.
It was shown in \cite{W2010} and  \cite{W2012} that ${\rm ISP}(\kappa)$ implies 
the two cardinal tree property ${\rm TP}(\kappa, \lambda)$ and the failure of the weak square principle 
$\square (\kappa,\lambda)$, for all regular $\lambda \geq \kappa$. 
The equivalence between ${\rm GM}(\kappa,\omega_1)$ and  ${\rm ISP}(\kappa)$ was established
in \cite{VW2011}, where it was also proved  that ${\rm ISP}(\omega_2)$ follows from the Proper Forcing Axiom. 
${\rm ISP}(\kappa)$ has strong influence on cardinal arithmetic. 
By using the results of Viale \cite{V2012},  Krueger \cite{Kru2019} showed that, for a regular cardinal $\kappa$,
 the principle ${\rm ISP}(\kappa)$, and hence ${\rm GM}(\kappa,\omega_1)$, implies the Singular Cardinal Hypothesis holds above $\kappa$. 
Let us also mention that in \cite{T2016} Trang showed the consistency of ${\rm GM}(\omega_3,\omega_2)$
assuming the existence of a supercompact cardinal.  In his model the Continuum Hypothesis holds.

We also recall the related notion of the  $\gamma$-approximation property, introduced by Hamkins in \cite{Hamkins2003}.
     
\begin{definition}\label{gamma-approximation-property}
Let $\gamma$ be an uncountable regular cardinal. Suppose $M$ and $N$ are transitive models (sets or classes), $M\subseteq N$ and $\gamma \in M$. 
We say the pair $(M,N)$ satisfies the $\gamma$-{\em approximation property}, if whenever $x\in N$ is a bounded subset of $M$
such that $x\cap a\in M$, for all $a\in M$ with $|a|^M <\gamma$, then $x\in M$. 
\end{definition}

\begin{remark} Suppose  $R$ is a powerful model and let $\gamma$ be a regular cardinal with $\gamma \in R$.  
Note that $M\prec R$ is a $\gamma$-guessing model if and only if the pair $(\overline{M}, V)$ satisfies the $\gamma$-approximation property. 
\end{remark}

Our plan is to strengthen the principle ${\rm GM}(\omega_2,\omega_1)$ in order to control the approachability ideal on $\omega_2$. 
Let us first recall the relevant definitions from \cite{Shelah1991}. 

\begin{definition}\label{appseq} Let $\lambda$ be a regular cardinal. A $\lambda$-{\em approaching sequence} is
a sequence of bounded subsets of $\lambda$. If $\bar a = (a_\xi: \xi < \lambda)$ is a $\lambda$-approaching sequence, 
we let $B(\bar a)$ denote the set of all $\delta < \lambda$ such that 
there is a cofinal subset $c\subseteq \delta$ such that:
\begin{enumerate}
\item ${\rm otp}(c) < \delta$, in particular $\delta$ is singular, 
\item for all $\gamma < \delta$, there exists $\eta < \delta$ such that $c \cap \gamma = a_\eta$.
\end{enumerate}
\end{definition}

\begin{definition}\label{AppId}
Suppose $\lambda$ is a regular cardinal. Let $I[\lambda]$ be the ideal generated by the sets $B(\bar a)$, 
for all $\lambda$-approaching sequences $\bar a$,
and the non stationary ideal ${\rm NS}_\lambda$.  
\end{definition}

\begin{remark} It is straightforward to check that $I[\lambda]$ is a normal ideal on $\lambda$, but it may be non proper.
$I[\lambda]$ is called the {\em approachability ideal} on $\lambda$. 
\end{remark}

Shelah asked if it is consistent that $I[\kappa^+]\restriction S_{\kappa^+}^{\kappa}= {\rm NS}_{\kappa^+}\rest S_{\kappa^+}^{\kappa}$,
for a regular cardinal $\kappa$.   Mitchell in \cite{MI2009} answered this question affirmatively by showing the following. 

\begin{theorem}[Mitchell, \cite{MI2009}]
Assume that $\kappa$ is a $\kappa^+$-Mahlo cardinal. Then there is a generic extension in which 
$I[\omega_2]\restriction S_{\omega_2}^{\omega_1}= {\rm NS}_{\omega_2}\rest S_{\omega_2}^{\omega_1}$.
\end{theorem}

\begin{remark}
In his paper \cite{MI2009}, Mitchell mentions that the large cardinal assumption used in his result is necessary
 by an unpublished theorem of Shelah; a proof can be found in \cite{MI2004}, Theorem 13. 
 It is also mentioned at the end of  \cite{MI2009} that one can prove the same result for $\lambda^{++}$ 
 where $\lambda$ is a regular cardinal, more precisely let $\lambda$ regular be given, 
 then under the same assumption (with $\kappa>\lambda$) of his theorem there is a generic extension of the universe 
 satisfying $I[\lambda^{++}]\rest S_{\lambda^{++}}^{\lambda^+} = {\rm NS}_{\lambda^{++}}\rest S_{\lambda^{++}}^{\lambda^+}$.
\end{remark}

We now formulate a principle that implies Mitchell's result. 

\begin{definition}[$\rm FS(\kappa^+,\gamma)$]\label{FS} Suppose $\gamma \leq \kappa$ are regular uncountable cardinals. 
The principle ${\rm FS}(\kappa^+,\gamma)$ asserts that, 
for every $X\in H_{\kappa^+}$, there is a collection $\mathcal G$ of $\gamma$-guessing models of cardinality $\kappa$ all containing $X$
such that $\{ M\cap \kappa^+: M\in \mathcal G \}$ is $\kappa$-closed and unbounded in $\kappa^+$. 
\end{definition}

\begin{proposition}
Suppose $\kappa$ is a regular uncountable cardinal and ${\rm FS}(\kappa^+,\kappa)$ holds. 
Then we have that $I[\kappa^+]\rest S_{\kappa^+}^\kappa= {\rm NS}_{\kappa^+}\rest S_{\kappa^+}^\kappa$. 
\end{proposition}

\begin{proof}
Suppose that ${\rm FS}(\kappa^+),\kappa$ holds. Let $\bar a=(a_\xi : \xi <\kappa^+)$ be a $\kappa^+$-approaching sequence. 
Let $\mathcal G$ be the family of $\kappa$-guessing models  all containing $\bar a$ whose existence is guaranteed by ${\rm FS}(\kappa^+)$. 
We show that  $M\cap \kappa^+ \notin B(\bar a)$, for any $M\in \mathcal G$ such that $\cof(M\cap \kappa^+)=\kappa$. 
Fix one such $M\in \mathcal G$. Let $\delta = M\cap \kappa^+$ and suppose $c\subseteq \delta$ satisfies 
(1) and (2) of \cref{appseq}. Let $\mu = {\rm otp}(c)$. Note that $\mu< \delta$, hence $\mu \in M$. 
Since $\bar a\in M$, we have that $c\cap \gamma \in M$, for all $\gamma <\delta$, and hence $c\cap Z\in M$, for all $Z\in M$ with $|Z|<\kappa$.
Since $M$ is a $\kappa$-guessing model,  there must be $d\in M$ such that $c\cap \delta = d\cap \delta$. We may assume $d\subseteq \kappa^+$.
Then $c$ is an initial segment of $d$, so if $\rho$ is the $\mu$-th element of $d$ then $d\cap \rho = c$.
Since $\mu, d \in M$, we have $\rho\in M$ as well, and hence $c=d \cap \rho \in M$. 
But then $\delta= \sup (c)$ belongs to $M$, a contradiction. 
\end{proof}

${\rm FS}(\omega_2,\omega_1)$  is a local principle, i.e. it refers only to $H_{\omega_3}$, therefore it cannot imply ${\rm GM}(\omega_2,\omega_1)$. 
We now formulate a principle that implies both ${\rm FS}(\omega_2,\omega_1)$ and ${\rm GM}(\omega_2,\omega_1)$.
We state it  for any pair of  uncountable regular cardinals $\gamma \leq \kappa$.

\begin{definition} Let   $\gamma \leq \kappa$  be regular cardinal cardinals. A model  $M$ of cardinality $\kappa^+$
is {\em strongly $\gamma$-guessing} if it is the union of an increasing chain $(M_\xi : \xi <\kappa^+)$
of $\gamma$-guessing models of cardinality $\kappa$ and $M_\xi= \bigcup \{ M_\eta : \eta < \xi\}$, for every $\xi$ of cofinality $\kappa$. 
\end{definition}

\begin{remark} We only define the notion of strongly $\gamma$-guessing for sets $M$ whose
cardinality is a successor of a regular cardinal $\kappa$. We have not explored other variations of this concept. 
Note that if $M$ is strongly $\gamma$-guessing then $M$ is $\gamma$-guessing. 
\end{remark}

For a powerful model $R$ and regular cardinals $\gamma$ and $\kappa$, we let 

\[
\mathfrak G^+_{\kappa^{++},\gamma}(R) = \{ M\in {\mathcal P}_{\kappa^{++}}(R): M\prec R  \mbox{ and } M \mbox{ is strongly $\gamma$-guessing} \}. 
\]

\begin{definition}\label{GM+}
For a powerful model $R$, ${\rm GM}^+(\kappa^{++},\gamma,R)$ states that  
${\mathfrak G}^+_{\kappa^{++},\gamma}(R)$ is stationary in ${\mathcal P}_{\kappa^{++}}(R)$.
${\rm GM}^+(\kappa^{++},\gamma)$ is the statement that ${\rm GM}^+(\kappa^{++},\gamma, H_\theta)$ holds,
 for all sufficiently large regular $\theta$. 
\end{definition}

Clearly ${\rm GM}^+(\kappa^{++},\gamma)$ implies both 
${\rm FS}(\kappa^+,\gamma)$ and ${\rm GM}(\kappa^+,\gamma)$.
The main result of this paper is the following.

\begin{theorem}\label{main theorem}
Suppose that $\kappa$ is supercompact and $\lambda > \kappa$ is inaccessible. Then there is a forcing notion such that 
in the generic extension  ${\rm GM}(\omega_2,\omega_1)$ and ${\rm FS}(\omega_2,\omega_1)$ hold. 
If, in addition, $\lambda$ is supercompact, then ${\rm GM}^+(\omega_3,\omega_1)$ holds as well. 
\end{theorem}

\begin{remark} By what we have said above the principle ${\rm GM}^+(\omega_3,\omega_1)$ implies the following: 
$I[\omega_2]\rest S_{\omega_2}^{\omega_1}= {\rm NS}_{\omega_2}\rest S_{\omega_2}^{\omega_1}$,
the principles ${\rm ISP}(\omega_2)$ and ${\rm ISP}(\omega_3)$, and hence
the tree tree property at $\omega_2$ and $\omega_3$, the Singular Cardinal Hypothesis, and the failure of the weak square
principle $\square (\omega_2,\lambda)$, for all regular $\lambda \geq \omega_2$. 
It is also not hard to see that it implies the negation of the weak Kurepa Hypothesis at $\omega_1$
and that the continuum is at least $\omega_3$. 
\end{remark}

The notion of strong properness, introduced by Mitchell in \cite{MI2005}, plays the key role in our construction. 
Let us recall the following definition. 

\begin{definition}\label{strong-generic} Let $\mathbb P$ be a forcing notion and $A$ a set. We say that $p\in \mathbb P$
is $(A,\mathbb P)$-{\em strongly generic} if for all $q\leq p$ there is a condition $q\rest A\in A$ such that any 
$r\leq  q\rest A$ with $r\in A$ is compatible with $q$. 
\end{definition}

\begin{definition}[Strong properness]
Let $\mathbb P$ be a forcing notion, and $\mathcal S$ a collection of sets. 
We say that $\mathbb P$ is $\mathcal S$-{\em strongly proper} if, for every $A\in \mathcal S$ and $p\in A\cap \mathbb P$,
there is $q\leq p$ that is $(A,\mathbb P)$-strongly generic. 
\end{definition}

The following  proposition connects the approximation property with strong properness.

\begin{proposition}\label{guessingbystronglyproper}
Let $\mathbb P$ be a forcing notion and $\kappa$ an uncountable regular cardinal. 
Suppose $\mathbb P$ is $\mathcal S$-strongly proper, for some stationary subset $\mathcal S$
of ${\mathcal P}_{\kappa}(\mathbb P)$. If $G$ is $V$-generic over $\mathbb P$, then $(V,V[G])$
has the $\kappa$-approximation property. 
\end{proposition}
\begin{proof}
Work in $V$. Let $\alpha$ be an ordinal, $\dot X$ a $\mathbb P$-name, and suppose some condition $p\in \mathbb P$
forces that $\dot X\subseteq \alpha$ and  $\dot X\cap \check Z\in V$, for all  $Z\in V$ with $|Z|^V< \kappa$.
Fix a sufficiently large regular cardinal $\theta$. By the stationarity of $\mathcal S$, we can find $M\prec H_{\theta}$
containing $p,\mathbb P, \dot A$, and such that $M\cap \mathbb P\in \mathcal S$. 
Let $q\leq p$ be $(M\cap \mathbb P)$-strongly generic. Since $M\cap \mathbb P$ is of size $<\kappa$, 
by strengthening $q$ if necessary, we may assume that $q$ decides $\dot X \cap M$. 
Since $q\rest (M\cap \mathbb P)$ and $p$ are compatible, and $M$ is elementary, they are compatible in $M$. 
Therefore, by replacing $q\rest (M\cap \mathbb P)$ by a stronger condition in $M$,  we may assume that it extends $p$. 
We now argue that $q\rest (M\cap \mathbb P)$ decides $\dot X$. Otherwise, by elementary of $M$, we
can find $\xi \in \alpha \cap M$ and $r_0,r_1\in M$ with $r_0, r_1\leq q\rest (M\cap \mathbb P)$ such that $r_0$ forces $\xi \in \dot X$
and $r_1$ forces $\xi \notin \dot X$. Now, by strong genericity of $q$, we have that $r_0$ and $r_1$ are both compatible
with $q$. Let $s_0$ be a common extension of $q$ and $r_0$, and $s_1$ a common extension of $q$ and $r_1$. 
Then $s_0, s_1\leq q$  and force contradictory information about $\xi \in \dot X$. 
This contradicts the fact that $q$ decides $\dot X \cap M$.
 \end{proof}

We will also need the following well-known theorem due to Magidor.

\begin{theorem}[Magidor, \cite{MA1971}]\label{factMagidor}
The following are equivalent for a regular cardinal $\kappa$.
\begin{enumerate}
\item $\kappa$ is supercompact.
\item For every $\gamma>\kappa$ and $x\in V_\gamma$ there exist $\bar\kappa < \bar\gamma<\kappa$, 
and an elementary embedding $j:V_{\bar\gamma}\rightarrow V_{\gamma}$ with critical point $\bar\kappa$ such 
that $j(\bar\kappa)=\kappa$ and $x\in j[V_{\bar\gamma}]$. 
\end{enumerate}
\end{theorem}
\begin{proof}[\nopunct] 
 \end{proof}

\section{Virtual Models}

In this section we review the notion of virtual models introduced in  \cite{Velisemiproper} and \cite{Veliproper}. 
In \cite{Veliproper} we  used virtual models of two types: countable and internally club (I.C.) models of size $\aleph_1$.
In the current situation we replace the I.C. models by models that have a much stronger closure property that we
call Magidor models. 
We shall consider the language $\mathcal L$ obtained by adding a single constant symbol $c$ to the standard language $\mathcal L_\epsilon$ of set theory.
Let us say that a structure $\mathcal A$ of the form $(A,\in, \kappa)$ is \emph{suitable}
if $A$ is a transitive set satisfying $\ZFC$ in the expanded language, 
when $\kappa$, the interpretation of the constant symbol $c$, is inaccessible in $A$. 
We shall often abuse notation and refer to the structure $(A,\in,\kappa)$ simply by $A$.
Suppose $\mathcal A$ is a suitable structure. If $\alpha$ is an ordinal in $A$,
we let $A_\alpha$ denote $A\cap V_\alpha$. Finally, we let
\[
E_{\mathcal A}=\{ \alpha \in A : A_\alpha \prec A \}. 
\]
Note that $E_A$ is a closed, possibly empty, subset of $\ORD^A$. It is not definable in $A$, but $E_A\cap \alpha$
is uniformly definable in $A$ with parameter $\alpha$, for each $\alpha \in E_A$. 
If $\alpha \in E_A$  we let ${\rm next}_{A}(\alpha)$ be the least ordinal in $E_A$ above $\alpha$,
if such an ordinal exists. Otherwise, we leave ${\rm next}_A(\alpha)$ undefined.
We start with a simple technical lemma.

\begin{lemma}\label{alphasharpinE}  
Suppose $M$ is an elementary submodel of a suitable structure $A$. Then
\begin{enumerate}
\item
If $\alpha \in E_A$ and $(M\cap \ORD^A)\setminus\alpha\neq\varnothing$, then ${\rm min}(M\cap\ORD^A\setminus\alpha)\in E_A$.
\item
$\sup ( E_{A}\cap M) = {\rm sup} (E_{A}\cap \sup (M\cap \ORD^A))$.
\label{1.1}
\end{enumerate}
\end{lemma}

\begin{proof} The second item follows from the first one, so we only give the proof of (1). 
Let $\beta$ be the least ordinal in $M\setminus \alpha$.
We need to show that $A_\beta$ is an elementary submodel of $A$. 
Suppose otherwise, then by the Tarski-Vaught criterion,
 there is a tuple $\bar{x}\in A_\beta$
and a formula $\varphi (y,\bar{x})$ such that $A \models \exists y \varphi(y,\bar{x})$,
but there is no $y\in A_\beta$ such that $A \models \exists y \varphi(y,\bar{x})$.
Since $\beta \in M$ and $M$ is an elementary submodel of $A$,
there is such a tuple $\bar{x}\in A_\beta\cap M$.
Now, $\beta$ is the least ordinal in $M$ above $\alpha$,
therefore $\bar{x}\in M\cap A_\alpha$.
Since $A_\alpha$ is an elementary submodel of $A$,
there is $y'\in A_\alpha$ witnessing that $A_\alpha \models \varphi(y',\bar{x})$
and so $A \models \varphi(y',\bar{x})$.
Since $\alpha \leq \beta$, it follows that $y'\in A_\beta$, a contradiction.
\end{proof}

\begin{definition}\label{M(X)-definition} Suppose $M$ is a submodel of a suitable
structure $A$ and $X$ is a subset of $A$. Let
$$
{\rm Hull}(M,X) =\{ f(\bar{x}) : f\in M, \bar{x}\in X^{<\omega}, f \mbox{ is a function, and } \bar{x}\in \dom (f)\}.
$$
\end{definition}

The main reason we have defined the $\rm Hull$ operation in this way is that it allows us to define
the Skolem hull of $M$ and $X$ without referring explicitly to the ambient model $A$.

\begin{lemma}\label{hull-elementary} Suppose $A$ is a suitable structure,
$M$ is an elementary submodel of $A$ and $X$ is a subset of $A$.
Let $\delta$ be $\sup (M\cap \ORD^A)$, and suppose $X\cap A_\delta$ is nonempty.
Then ${\rm Hull}(M,X)$ is the least elementary submodel of $A$ containing $M$ and $X\cap A_\delta$.
\end{lemma}

\begin{proof} For each $\gamma \in A$, let ${\rm id}_\gamma$ be the identity function
on $A_\gamma$. Clearly, if $\gamma \in M$ then ${\rm id}_\gamma \in M$. Therefore,
$X\cap A_\delta$ is a subset of ${\rm Hull}(M,X)$. Let $\gamma \in M$ be such that $X\cap A_\gamma$ is nonempty.
For each $z\in M$, the constant function $c_z$ defined on $A_\gamma$ is in $M$,
therefore $M$ is a subset of ${\rm Hull}(M,X)$. The minimality of ${\rm Hull}(M,X)$ is clear from the definition.
It remains to show that ${\rm Hull}(M,X)$ is an elementary submodel of $A$.
We check the Tarski-Vaught criterion for ${\rm Hull}(M,X)$ and $A$. Let $\varphi$ be a formula and
$a_1,\ldots,a_n \in {\rm Hull}(M,X)$ such that $A\models \exists u \varphi(u,a_1,\ldots,a_n)$.
Then we can find functions $f_1,\ldots, f_n\in M$ and tuples $\bar{x}_1,\ldots,\bar{x}_n \in X^{<\omega}$
such that $a_i=f_i(\bar{x}_i)$, for all $i$. If $D_i$ is the domain of $f_i$, this implies that $\bar{x}_i\in D_i$.
By regularity and the axiom of choice in $A$ we can find a function $g$ defined on $D_1 \times \ldots \times D_n$
such that for every $\bar{y}_1\in D_1,\ldots,\bar{y}_n \in D_n$, if there is $u$ such that
$A\models \varphi (u, f_1(\bar{y}_1),\ldots,f_n(\bar{y}_n))$ then
$g(\bar{y}_1,\ldots,\bar{y}_n)$ is such a $u$. Moreover, by elementarity of $M$,
we may assume that $g\in M$. Let $a =g(\bar{x}_1,\ldots,\bar{x}_n)$.
It follows that $a\in {\rm Hull}(M,X)$ and $A \models \varphi (a,a_1,\ldots,a_n)$.
Therefore, ${\rm Hull}(M,X)$ is an elementary submodel of $A$.
\end{proof}

Now, let us fix an inaccessible cardinal  $\kappa$, a cardinal $\lambda >\kappa$ such that 
$V_\lambda$ satisfies $\ZFC$.   We shall write
 $E$ instead of $E_{V_\lambda}$ and ${\rm next}(\alpha)$ instead of ${\rm next}_{V_\lambda}(\alpha)$.
For each $\alpha \in E$, we shall define certain families $\mathscr F_\alpha\in V_\lambda$, 
as well as relations $R_\alpha$ and operations $O_\alpha$ on $V_\lambda$. 
Being a member of $\mathscr F_\alpha$ will be expressed by a $\Sigma_1$-formula with parameter $V_\alpha$
and similarly for $R_\alpha$ and $O_\alpha$.
If $A$ is another suitable structure we can interpret these formulas in $A$ and obtain
families $\mathscr F_\alpha^A$, relations $R_\alpha^A$ and operations $O_\alpha^A$. In this section we shall only consider
suitable $A$ such that the interpretation of the constant symbol $c$ is $\kappa$ and $A\subseteq V_\lambda$. 
Note that if we have such an $A$ and $\alpha \in E_A \cap E$ with $A_\alpha = V_\alpha$ then $\mathscr F_\alpha^A \subseteq \mathscr F_\alpha$.
Similarly, if $x,y\in A$ are such that $xR_\alpha^A y$ then $xR_\alpha y$, and if $x\in A$ then $O_\alpha^A(x)=O_\alpha(x)$. 
If $M$ is an elementary submodel of a suitable $A$ and $\alpha\in E_A \cap M$ we shall write 
$\mathscr F_\alpha^M$ for $\mathscr F_\alpha^A\cap M$, and $R_\alpha^M$ and $O_\alpha^M$  for restrictions of  $R_\alpha^A$ and $O_\alpha^A$ to $M$.

\begin{definition}\label{A-def} Suppose $\alpha \in E$. We let ${\mathscr A}_\alpha$ denote the set of all
transitive $A$ that are elementary extensions of $V_\alpha$ and have the same cardinality as $V_\alpha$.
\end{definition}

Note that if $A \in \mathscr A_\alpha$ and $\alpha \in A$ then $E_A\cap \alpha=E\cap \alpha$.
If $A\in \mathscr A_\alpha$ we will refer to $V_\alpha$ as the {\em standard part} of $A$.
Note that if $A$ has nonstandard elements then $\alpha \in E_A$. 


\begin{definition}
Suppose $\alpha \in E$. We let $\mathscr{V}_\alpha$ denote the collection of all substructures $M$ of $V_\lambda$ of
size less than $\kappa$ such that,
 if we let $A={\rm Hull}(M,V_\alpha)$, then $A\in \mathscr A_\alpha$ and $M$ is an elementary 
submodel of $ A$.
\end{definition}

\begin{definition} We refer to the members of $\mathscr{V}_\alpha$ as the $\alpha$-{\em models}. 
 We write $\mathscr{V}_{< \alpha}$ for $\bigcup \{\mathscr{V}_\gamma : \gamma \in E \cap \alpha\}$.
 Collections $\mathscr{V}_{\leq \alpha}$ and $\mathscr{V}_{\geq \alpha}$ are defined in the obvious way. 
 We will write $\mathscr V$ for $\mathscr{V}_{< \lambda}$ . 
 If $M\in \mathscr V$, we then write  $\eta(M)$ for the unique ordinal $\alpha$ such that $M\in \mathscr V_\alpha$.
\end{definition}

\begin{remark} Note that if $M \in \mathscr V_\alpha$ then $\sup(M \cap \ORD)\geq \alpha$.
In general, $M$ is not elementary in $V_\lambda$, in fact, this only happens if $M \subseteq V_\alpha$.
In this case we will say that $M$ is a {\em standard $\alpha$-model}.
\end{remark}

\begin{convention} We refer to members of $\mathscr V$ as
 {\em virtual models}. We also refer to members of $\mathscr V^A$, for some
suitable $A$ with $A\subseteq V_\lambda$, as {\em general virtual models}.
\end{convention}

\begin{definition}\label{alpha-equiv-def} Suppose $M,N\in \mathscr V$ and $\alpha \in E$.
We say that an isomorphism $\sigma:M\rightarrow N$ is an $\alpha$-{\em isomorphism}
if it has an extension to an isomorphism $\bar{\sigma}:{\rm Hull}(M,V_\alpha)\rightarrow {\rm Hull}(N,V_\alpha)$.
We say that $M$ and $N$ are $\alpha$-{\em isomorphic} and write $M\cong_{\alpha}N$ if
there is an $\alpha$-isomorphism between them. Note that if  $\sigma$ and $\bar\sigma$ exist, they are unique.
\end{definition}

Clearly, $\cong_{\alpha}$ is an equivalence relation, for every $\alpha\in E$. Note that if $M \in \mathscr V_\gamma$,
for some $\gamma < \alpha$, then the only model $\alpha$-isomorphic to $M$ is $M$ itself.
Suppose  $\alpha,\beta \in E$ and $\alpha \leq \beta$.
It is easy to see that, if $M,N\in \mathscr V$ are $\beta$-isomorphic,
then they are $\alpha$-isomorphic.  We will now see that, if $\alpha <\beta$, then
for every $\beta$-model $M$ there is a canonical representative of the $\cong_\alpha$-equivalence
class of $M$ which is an $\alpha$-model.

\begin{definition}\label{projection-models-def} Suppose $\alpha$ and $\beta$ are members of  $E$
and $M$ is a $\beta$-model. Let $\overline{{\rm Hull}(M,V_\alpha)}$ be the transitive collapse of
${\rm Hull}(M,V_\alpha)$, and let $\pi$ be the collapsing map. We define  $M \! \restriction \! \alpha$ to be $\pi[M]$,
i.e. the image of $M$ under the collapsing map of ${\rm Hull}(M,V_\alpha)$.
\end{definition}

\begin{remark} Note that if $\beta < \alpha$ then $M\restriction \alpha = M$.
If $\beta \geq \alpha$, then $\overline{{\rm Hull}(M,V_\alpha)}$ belongs to $\mathscr A_\alpha$, so $M\! \restriction \! \alpha$ is
an $\alpha$-model which is $\alpha$-isomorphic to $M$. Note also that if $\beta=\alpha$, then
$M\! \restriction \! \alpha =M$ since ${\rm Hull}(M,V_\alpha)$ is already transitive.

Note also that if $A\in \mathscr A_\alpha$ then $\mathscr V_\alpha^A \subseteq \mathscr V_\alpha$. 
Therefore, if $A,B\in \mathscr A_\alpha$, $M \in \mathscr V^{A}$, and $N \in \mathscr V^B$, we can 
still write $M \cong_\alpha N$ if $M \restriction \alpha = N \restriction \alpha$. This is 
of course equivalent to the existence of an $\alpha$-isomorphism between $M$ and $N$.
\end{remark}

The following is straightforward.

\begin{proposition}\label{transitivity-projections} Suppose $\alpha,\beta \in E$ and
$\alpha \leq \beta$. Let $M\in \mathscr V$. Then
$(M \! \restriction \! \beta)\! \restriction \! \alpha = M\! \restriction \! \alpha$.
\end{proposition}
\begin{proof}[\nopunct] 
 \end{proof}

We also need to define a version of the membership relation, for every $\alpha$ in $E$.

\begin{definition}\label{membership-alpha-def} Suppose $M,N\in \mathscr V$  and $\alpha \in E$. 
We write $M\in_\alpha N$ if there is $M'\in N$ with $M' \in \mathscr V^N$ such that $M'\cong_\alpha M$. 
If this happens, we say that $M$ is $\alpha$-in $N$.
\end{definition}

Note that if $M\subseteq V_\alpha$, this simply means that $M\in N$. However, in general,
we may have $M \in_\alpha N$ even if the rank of $M$ is higher than the rank of $N$.
We shall often use the following simple facts without mentioning them. 

\begin{proposition}\label{trans} 
Suppose $M, N \in \mathscr V$ with $M\in N$. Let $\alpha \in E$, and suppose 
$N'\in \mathscr V^A$, for some $A\in \mathscr A_\alpha$,  and $\sigma : N \rightarrow N'$ is an $\alpha$-isomorphism. 
Then $M$ and $\sigma(M)$ are $\alpha$-isomorphic. 
\end{proposition}

\begin{proof} Since $| M|  < \kappa < | V_\alpha|$, we conclude that $M\subseteq {\rm Hull}(N,V_\alpha)$, 
and hence ${\rm Hull}(M,V_\alpha) \subseteq {\rm Hull}(N,V_\alpha)$. Let $\bar\sigma$ be the extension 
of $\sigma$ to ${\rm Hull}(N,V_\alpha)$. 
 It follows that $\bar\sigma \restriction {\rm Hull}(M,V_\alpha)$ is
an isomorphism between ${\rm Hull}(M,V_\alpha)$ and ${\rm Hull}(\sigma(M),V_\alpha)$.
Hence $\bar\sigma \restriction M$ is an $\alpha$-isomorphism between $M$ and $\sigma(M)$.
\end{proof}

\begin{proposition}\label{projection-membership} Let $\alpha,\beta \in E$ with $\alpha \leq \beta$.
Suppose $M,N \in \mathscr V_{\geq \beta}$ and $M\in_\beta N$. 
Then $M\! \restriction \! \alpha \in_\alpha N\! \restriction \! \alpha$.
\end{proposition}
\begin{proof}
Fix some $M'\in N$ such that $M\cong_\beta M'$. Since $\alpha \leq \beta$, we have that $M\cong_\alpha M'$. 
If $\pi$  is the Mostowski collapse map of ${\rm Hull}(N,V_\alpha)$, then $\pi(M')\in N\restriction\alpha$. 
On the other hand, since $| M | < \kappa <| V_\alpha|$, we have that ${\rm Hull}(M',V_\alpha) \subseteq {\rm Hull}(N,V_\alpha)$ and
$\pi [M']= \pi (M')$. It follows that $\pi \restriction {\rm Hull}(M',V_\alpha)$ is an isomorphism between ${\rm Hull}(M',V_\alpha)$
and ${\rm Hull}(\pi(M'),V_\alpha)$. 
Therefore, $M\cong_\alpha\pi(M')\in N\restriction\alpha$.
 \end{proof}

We refer to the following proposition as the continuity of the $\alpha$-isomorphism.
\begin{proposition}\label{isocont}
Let $\alpha$ be a limit point of $ E$. Suppose $N,M\in \mathscr V$ and $M\cong_\gamma N$ for unboundedly many $\gamma$ below $\alpha$.
Then $M\cong_\alpha N$.
\end{proposition}
\begin{proof}
For each $\gamma \in E \cap \alpha$, let $\sigma_\gamma$ be the unique isomorphism between ${\rm Hull}(M,V_\gamma)$ and ${\rm Hull}(N,V_\gamma)$
such that $\sigma_{\gamma}[M]=N$. If $\gamma < \gamma'$, we have that $\sigma_{\gamma'}\restriction {\rm Hull}(M,V_\gamma)= \sigma_{\gamma}$.
Let $\sigma = \bigcup \{\sigma_\gamma : \gamma \in E\cap \alpha \}$. Then $\sigma$ witnesses that $M$ and $N$ are $\alpha$-isomorphic. 
\end{proof}

\begin{proposition}\label{incontcount}
Let $\alpha$ be a limit point of $E$ of uncountable cofinality. Assume that $M,N\in \mathscr V$ and $N$ is countable. 
Suppose that $M\in_\gamma N$ for unboundedly many $\gamma<\alpha$. Then $M\in_\alpha N$.
\end{proposition}
\begin{proof}
Since $N$ is countable and $\alpha$ is of uncountable cofinality, there is $M'\in N$ with $M'\in \mathscr V^N$ such that $M\cong_\gamma M'$, for unboundedly many $\gamma \in E\cap \alpha$. 
By  \cref{isocont} we have that $M \cong_\alpha M'$, and hence $M\in_\alpha N$. 
\end{proof}

In our forcing we will use two types of virtual models, the countable ones and  some nice models of size less than $\kappa$ defined below.

\begin{definition}
For $\alpha \in E$, we let $\mathscr C_\alpha$ denote the collection of countable models in  $\mathscr V_\alpha$. 
We define similarly $\mathscr C_{<\alpha}$, $\mathscr C_{\leq\alpha}$ and $\mathscr C_{\geq \alpha}$. We write $\mathscr C$ for $\mathscr C_{<\lambda}$, and
 $\mathscr C_{\rm st}$ for the collection of standard models in $\mathscr C$.
\end{definition}


\begin{proposition}\label{cstationaryinV}
Suppose $\lambda$ is of uncountable cofinality. Then $\mathscr C_{\rm st}$ contains a club in ${\mathcal P}_{\omega_1}(V_\lambda)$.
\end{proposition}

\begin{proof} First note that since $\lambda$ is of uncountable cofinality $E$ is unbounded  and thus club in $\lambda$.
Suppose $M$ is a countable elementary submodel of $(V_\lambda, \in, E)$. Let $\alpha = \sup (M \cap E)$. 
Note that $M\cap {\rm ORD}$ is unbounded in $\alpha$. Hence $M$ is a standard $\alpha$-model.
\end{proof}

The following definition is motivated by Magidor's reformulation of supercompactness, see \cref{factMagidor}.

\begin{definition}\label{Magidor models}
We say that  a model $M$  is a $\kappa$-{\em Magidor model} if, letting $\overline{M}$ be the transitive collapse of $M$ and $\pi$ the collapsing map, 
$\overline{M}=V_{\bar\gamma}$, for some $\bar\gamma <\kappa$ with ${\rm cof}(\bar\gamma) \geq \pi (\kappa)$, and $V_{\pi(\kappa)}\subseteq M$. 
If $\kappa$ is clear from the context we omit it. 
\end{definition}

\begin{proposition}\label{Magidor-stationary} Suppose $\kappa$ is supercompact and $\mu >\kappa$ with ${\rm cof}(\mu) \geq \kappa$. 
Then the set of $\kappa$-Magidor models is stationary in $\mathcal P_{\kappa}(V_\mu)$.
\end{proposition}

\begin{proof} 
Fix a function $F: [V_\mu]^{< \omega}\rightarrow V_\mu$. We have to find a $\kappa$-Magidor model closed under $F$. 
Let $\gamma>\mu$ be such that $V_\gamma$ satisfies a sufficient fragment of $\rm ZFC$. 
Since $\kappa$ is supercompact, by \cref{factMagidor} we can find $\bar\kappa < \bar\gamma <\kappa$ and
an elementary embedding $j: V_{\bar\gamma}\rightarrow V_\gamma$ with critical point $\bar\kappa$ such that $j(\bar\kappa)=\kappa$
and such that $F\restriction [V_\mu]^{<\omega} \in j [V_{\bar\gamma}]$. Note that this implies that $V_\mu \in  j [V_{\bar\gamma}]$.
Let $\bar\mu$ be such that $j(\bar\mu)= \mu$.
Since ${\rm cof}(\mu) \geq \kappa$, by elementarity we must have that ${\rm cof}(\bar\mu)\geq \bar\kappa$. 
Let $N = j [V_{\bar\mu}]$. Then $N$ is a $\kappa$-Magidor elementary submodel of $V_\mu$ that is closed under $F$, as required. 
\end{proof} 

\begin{definition} Let $\mathscr U^\kappa_{\alpha}$ be the collection of all $M\in \mathscr V_{\alpha}$ that are $\kappa$-Magidor models. 
We define $\mathscr U^\kappa_{<\alpha}$, $\mathscr U^\kappa_{\leq\alpha}$, and $\mathscr U^\kappa_{\geq \alpha}$ in the obvious way. 
We write $\mathscr U^\kappa$ for $\mathscr U^\kappa_{< \lambda}$. When $\kappa$ is clear from the context, we omit it.
We also write $\mathscr U_{\rm st}$ for the standard models in $\mathscr U$.
\end{definition}

\begin{remark} Suppose $M$ is a $\kappa$-Magidor $\alpha$-model. Let $V_{\bar\gamma}$ be its transitive collapse, and let $j$ be the inverse of the collapsing map $\pi$.  
Let $\bar{\kappa}= \pi (\kappa)$, and let $A= {\rm Hull}(M,V_\alpha)$. Note that $j:V_{\bar\gamma}\rightarrow A$ is an elementary embedding with critical point $\bar\kappa$
and $j(\bar\kappa)= \kappa$. 
\end{remark}

By \cref{Magidor-stationary} we have the following immediate corollary. 

\begin{proposition}\label{ustationaryinV}
Suppose $\kappa$ is supercompact and $\lambda$ is inaccessible. Then $\mathscr U_{\rm st}$ is stationary in $\mathcal P_{\kappa}(V_\lambda)$.
\end{proposition}

\begin{proof}[\nopunct] 
 \end{proof}
 

Note that  both classes $\mathscr C$ and $\mathscr U$ of virtual models are closed under projections. 
We shall study some particular finite collections of these two types of models. 
We start by establishing the following easy fact.

\begin{proposition}\label{intermediate}
Let $\alpha \in E$. Suppose $M,N,P\in \mathscr V$ and $M\in_\alpha N\in_\alpha P$. 
If either $N$ is countable or $P$ is a Magidor model then $M\in_\alpha P$. 
\end{proposition}
\begin{proof}
Pick $N'\in P$ with $N'\in \mathscr V^P$ which is $\alpha$-isomorphic to $N$. We first establish that $N'\subseteq P$. 
If $N$ is countable this is immediate. Suppose both $N$ and $P$ are Magidor models.
Let $\overline{N'}$ be the transitive collapse of $N'$, and let $\pi$ be the collapsing map. 
 Then $\overline{N'}\in V_\kappa \cap P$ since $| N '| < \kappa$. Since $P$ is a Magidor model,
we know that $V_\kappa \cap P$ is transitive, and hence $\overline{N'}\subseteq P$, but then also $N'\subseteq P$. 
Let $\sigma$ be an $\alpha$-isomorphism between $N$ and $N'$, and let $M'\in  N$ with $M'\in \mathscr V^N$
be a model that is $\alpha$-isomorphic to $M$. By \cref{trans} we know that $\sigma(M')$
is $\alpha$-isomorphic to $M'$, and also to $M$ by the transitivity of $\cong_\alpha$.
On the other hand $\sigma(M')\in N' \subseteq P$ and thus $M\in_\alpha P$, as desired. 
\end{proof}

Our next goal is to define when a virtual model $M$ is active at some $\alpha \in E$. 

\begin{definition}\label{def-active}
Let $M\in \mathscr V$. We say that $M$ is {\em active} at $\alpha \in E$ if 
$\eta (M)\geq \alpha$ and ${\rm Hull} (M,V_{\kappa_M})\cap E \cap \alpha$ is unbounded in $E\cap \alpha$,
where $\kappa_M = \sup (M\cap \kappa)$. We say that $M$ is {\em strongly active} at $\alpha$ if $\eta(M)\geq \alpha$
and $M\cap E\cap \alpha$ is unbounded in $E\cap \alpha$. 
\end{definition}

\begin{remark} We are primarily interested in the case $M \in \mathscr C \cup \mathscr U$.
First note that if $M$ is a Magidor model, then $V_{\kappa_M}\subseteq M$, hence $M$ is active at some $\alpha \in E$
if and only if it is strongly active at $\alpha$. 
The situation is quite different for countable models. 
If $M$ is countable, then the set of $\alpha \in E$ at which $M$ is strongly active is at most countable,
while the set of $\alpha \in E$ at which $M$ is active can be of size $|V_{\kappa_M}|$. 
One feature of our definition is that if $N\in_\alpha M$, then for all $\gamma \in E\cap \alpha$,
if $N$ is active at $\gamma$ then so is $M$.

Let us also remark what happens at levels $\alpha$ that are successor points of $E$.
Suppose  $\alpha = {\rm next}(\beta)$, for some $\beta \in E$, and $M$ is active at $\alpha$.
 We must have $\beta \in M$ as $\beta = \max (E\cap \alpha)$. 
we must also have $\sup (M\cap {\rm ORD})\geq \alpha$ since $\eta(M)\geq \alpha$,. If $\sup (M\cap {\rm ORD})= \alpha$
then $M$ is a countable standard model. If $\sup (M\cap {\rm ORD})>  \alpha$, let 
$\gamma = \min (M \cap {\rm ORD}\setminus \alpha)$, and let $A={\rm Hull}(M,V_\gamma)$. 
Then by \cref{alphasharpinE} $\gamma \in E_A$. 
Since $\gamma \in M$, we have that $E_A\cap (\gamma +1)\in M$ and therefore
we can compute $\alpha$ in $M$ as the the next element of $E_A\cap (\gamma +1)$ above $\beta$. 
Thus, in this case we have $\alpha \in M$. 
\end{remark}

It will be convenient to also have the following definition. 

\begin{definition}\label{a(M)} Suppose $M\in \mathscr V$. Let $a(M)= \{ \alpha \in E : M \mbox{ is active at } \alpha \}$ and $\alpha (M)=\max a(M)$.
\end{definition}

Note that $a(M)$ is a closed subset of $E$ of size at most $| {\rm Hull}(M,V_{\kappa_M})|$.

\begin{proposition}\label{active-Magidor} Let $M\in \mathscr V$ and $N\in \mathscr U$. Suppose $\alpha \in E$, $M$ and $N$ are active at $\alpha$,
and $M\in_\alpha N$. Then $\alpha \in N$.
\end{proposition}

\begin{proof} We may assume that $M$ and $N$ are $\alpha$-models. Let $A= {\rm Hull}(N,V_\alpha)$. Then $A\in \mathscr A_\alpha$. 
Fix $M^*\in N$ with $M^*\in \mathscr V^A$ which is $\alpha$-isomorphic to $M$.
Since $M^*$ is $\alpha$-equivalent to $M$, we have that $\alpha \in a^A(M^*)$. 
On the other hand, $a^A(M^*)\in N$ and has size $<\kappa_N$, hence $a^A(M^*)\subseteq N$.
It follows that $\alpha \in N$. 
\end{proof}

\begin{proposition}\label{incont} Let $M\in \mathscr V$ and $N\in \mathscr U$. 
Suppose $\alpha \in a(M)$ is a limit point of $E$ and $M \in_\gamma N$, for all $\gamma \in E \cap \alpha$. 
Then $\alpha \in N$ and $M \in_\alpha N$. 
\end{proposition}

\begin{proof}
Let $a = a(M)\cap N \cap \alpha$. Note that $A$ is unbounded in $\alpha$ and has size $< \kappa_N$. Since $N$ is closed under $<\kappa_N$-sequences,
it follows that $a\in N$, and hence $\alpha = \sup (a) \in N$. For $\gamma < \alpha$, let $M_\gamma = M\restriction \gamma$.
For $\gamma \in a$, we have that  $M \in_{\gamma}N$, and hence $M_\gamma \in N$. 
Let $A_\gamma = {\rm Hull}(M_\gamma,V_\gamma)$. For $\gamma, \delta \in a$ with $\gamma < \delta$, we have
that $M_\delta \restriction \gamma = M_\gamma$. In other words, $A_\gamma$ is the transitive collapse of ${\rm Hull}(M_\delta,V_\gamma)$
, and if $\sigma_{\gamma,\delta}$ is the inverse of the collapsing map, we have $\sigma_{\gamma,\delta}[M_\gamma]=M_\delta$. 
Each of the maps $\sigma_{\gamma,\delta}$ is definable from $M_\delta$ and $\gamma$, and hence it belongs to $N$.
Now, $N$ is closed under $<\kappa_N$-sequences and therefore the whole system $(A_\gamma, \sigma_{\gamma,\delta} : \gamma \leq \delta \in a)$ belongs to $N$.
Let $A$ be the direct limit of this system, and let $\sigma_\gamma$ be the canonical embedding of $A_\gamma$ to $A$.
If we let $\pi_\gamma$ be the collapsing map of ${\rm Hull}(M,V_\gamma)$ to $A_\gamma$, we then have that, for every $\gamma < \delta$, 
the following diagram commutes:

\[\begin{tikzcd}
{\rm Hull}(M,V_\gamma) \arrow{r}{{\rm id}} \arrow[swap]{d}{\pi_\gamma} & {\rm Hull}(M,V_\delta) \arrow{d}{\pi_\delta}  \\
A_\gamma  \arrow{r}{\sigma_{\gamma,\delta}} & A_\delta 
\end{tikzcd}
\]

\noindent Since ${\rm Hull}(M,V_\alpha) = \bigcup \{ {\rm Hull}(M,V_\gamma): \gamma \in a \}$, we have 
that $A$ is isomorphic to ${\rm Hull}(M,V_\alpha)$. 
Therefore, its transitive collapse is $A_\alpha={\rm Hull}(M\restriction \alpha,V_\alpha)$, and if we let $\pi$ be the collapsing map, 
$\pi [M]= M\restriction \alpha$. We can therefore identify $A$ with $A_\alpha$, and we get that $\sigma_\gamma [M_\gamma]=M\restriction \alpha$, 
for any $\gamma \in a$. Thus $M\restriction \alpha \in N$, as required. 
\end{proof}

We now  define an operation that will play the role of intersection for virtual models. 
We call it the {\em meet}. We only define the meet of two models of different types. 
Suppose $N\in \mathscr U$ and $M\in \mathscr C$. Let $\overline{N}$ be the transitive collapse of $N$,
and let $\pi$ be the collapsing map. Note that if $\overline{N}\in M$, then $\overline{N}\cap M$ is a countable elementary submodel 
of $\overline{N}$. Then $\overline{N}\cap M \in \overline{N}$ since $\overline{N}$ is closed under countable sequence.
Note that $\pi^{-1}(\overline{N}\cap M)= \pi^{-1}[\overline{N}\cap M]$, and this model is elementary in $N$.

\begin{definition} Suppose $N\in \mathscr U$ and $M \in \mathscr C$. Let $\alpha = \max (a(N)\cap a(M))$. 
We will define $N \land M$ if $N\in_\alpha M$. Let $\overline{N}$ be the transitive collapse of $N$, and let $\pi$
be the collapsing map. 
Set
\[ 
\eta = \sup (\sup (\pi^{-1}[\overline{N}\cap M]\cap {\rm ORD})\cap E\cap (\alpha +1)).
\]
We define the meet of $N$ and $M$ to be $N\land M=\pi^{-1}[\overline{N}\cap M]\restriction\eta$.

\end{definition}
To make sense of the above definition, we need to prove the following.

\begin{proposition}
Under the assumptions of the above definition, $N\land M\in \mathscr C_\eta$.
\end{proposition}
\begin{proof}
Since $\eta(N)\geq \alpha$ we can form the model $A={\rm Hull}(N,V_\alpha)$ and,  we therefore have $N \prec A$
and $V_\alpha \prec A$. 
 Since $\overline{N}\in M$, we have that $\overline{N}\cap M \prec \overline{N}$.
Therefore, we have $\pi^{-1}[\overline{N}\cap M] \prec N$. 
Now, $\eta \in E \cap (\alpha +1)$ and so $V_\eta \prec V_\alpha\prec A$. 
Moreover, $\sup (\pi^{-1}[\overline{N}\cap M]\cap {\rm ORD}) \geq \eta$. 
By \cref{hull-elementary} we have that ${\rm Hull}(\pi^{-1}[\overline{N}\cap M],V_\eta)\prec A$ and 
$V_\eta \subseteq {\rm Hull}(\pi^{-1}[\overline{N}\cap M],V_\eta)$.
It follows that the transitive collapse of ${\rm Hull}(\pi^{-1}[\overline{N}\cap M],V_\eta)$ belongs to $\mathscr A_\eta$
and thus the image of $\pi^{-1}[\overline{N}\cap M]$ under the collapsing map belongs to $\mathscr C_\eta$. 
\end{proof}

\begin{proposition}\label{meet-trace} Let $N\in \mathscr U$ and $M\in \mathscr C$. Suppose $\alpha \in E$ and the meet $N\land M$ is defined and active at $\alpha$.
Then $(N \land M)\cap V_\alpha = N\cap M \cap V_\alpha$. 
\end{proposition}
\begin{proof} Let $\beta = \max (a(N)\cap a(M))$. Since the meet of $N$ and $M$ is defined we must have $N \in_\beta M$. 
Since $N\land M$ is active at $\alpha$, we must have $\alpha \leq \beta$. 
Let  $N'\in \mathscr V^M$ be such that $N'\cong_\beta N$. Let $\sigma$ be the $\beta$-isomorphism between 
$N$ and $N'$. Notice that $\sigma$ is the identity on $N\cap V_\beta$ and thus also on $N\cap V_\alpha$.
 Let $\overline{N}$ denote the common transitive collapse of $N$ and $N'$, and let $\pi$ and $\pi'$ be the collapsing maps. 
Then the following diagram commutes.

\[
\begin{tikzcd}[row sep=2.5em]
N \arrow{dr}{\pi} \arrow{rr}{\sigma} &&N' \arrow[swap]{dl}{\pi'} \\
 & \overline{N} 
\end{tikzcd}
\]

Note that $(N\land M)\cap V_\alpha = \pi^{-1}[\overline{N}\cap M]\cap V_\alpha = \sigma^{-1}[N\cap M\cap V_\alpha]$.
Since $\sigma$ is the identity on $N \cap V_\alpha$, it follows that $(N \land M)\cap V_\alpha = N\cap M \cap V_\alpha$. 
\end{proof}

\begin{proposition}\label{meet-strongly-active} Let $\alpha \in E$. Suppose $N\in \mathscr U$ and $\M\in \mathscr C$, the meet $N\land M$ is defined, 
 and $N$ and $M$ strongly active at $\alpha$. Then $N\land M$ is strongly active at $\alpha$. 
\end{proposition}
\begin{proof} Let $\beta= \max (a(N)\cap a(M))$. Since both $N$ and $M$ are active at $\alpha$, we must have $\alpha \leq \beta$. 
Let $N'\in M$  with $N'\in \mathscr V^M$ be such that $N'\cong_\beta N$. 
Let $\sigma$ be the $\beta$-isomorphism between $N'$ and $N$.  Then $\sigma \rest N'\cap V_\beta$ is the identity. 
Note that $N'\cap M\cap E \cap \alpha$ is unbounded in $E\cap \alpha$.
Since $N'\cap V_\alpha = N\cap V_\alpha$,  we must have that $N\cap M\cap E \cap \alpha$ is also unbounded in $E\cap \alpha$. 
By \cref{meet-trace}, $N\land M$ is strongly active at $\alpha$. 
\end{proof}

The next proposition states the meet operation commutes with projections. 

\begin{proposition}\label{MPI}
Let $N\in \mathscr U$ and $M\in \mathscr C$. Suppose $\alpha \in E$ and the meet $N\land M$ is defined and active at $\alpha$.
Then $(N \land M)\restriction \alpha = N \restriction \alpha \land M \restriction \alpha$. 
\end{proposition}
\begin{proof}
First note that if $N\land M$ is active at $\alpha$, then $\alpha \in a(N) \cap a(M)$. It follows that $\alpha$ is the maximum of 
$a(N\restriction \alpha) \cap a(M\restriction \alpha)$.
Then note that $N\land M$ depends only on $\max (a(N)\cap a(M))$, $N$, and $M\cap \overline{N}$, where $\overline{N}$ is the transitive collapse of $N$. 
Now, $\overline{N}$ is also the transitive collapse of $N\restriction \alpha$. In fact, if $\sigma$ is the $\alpha$-isomorphism between $N$ and $N\restriction \alpha$,
and $\pi$ and $\pi'$ are the collapsing maps of $N$ and $N'$ respectively, then $\pi = \pi ' \circ \sigma$. Therefore, 
$\sigma\restriction \pi^{-1}[\overline{N}\cap M]$ is an an $\alpha$-isomorphism between $\pi^{-1}[\overline{N}\cap M]$
 and $\pi'^{-1}[\overline{N}\cap M\restriction \alpha]$. It follows that $(N \land M)\restriction \alpha = N \restriction \alpha \land M \restriction \alpha$. 
\end{proof}

\begin{proposition}\label{meetactiveness}
Let $\alpha \in E$. Suppose $N\in \mathscr U$, $M\in \mathscr C$, both are active at $\alpha$ and $N\in_\alpha M$.
Let $P$ be another virtual model also active at $\alpha$. Then $P \in_\alpha N \land M$ if only if $P\in_\alpha N$ and $P\in_\alpha M$.
\end{proposition}

\begin{proof}
By \cref{MPI} we may assume that $N,M$ and $P$ are all $\alpha$-models. 
Assume first that $P \in_\alpha N\land M$. In particular this means that $N \land M$ is active at $\alpha$. 
In particular we have that $N\land M \subseteq N$, and hence $P \in_\alpha N$. 
Fix $N' \in \mathscr V^M$ which is $\alpha$-isomorphic to $N$. Let $\overline{N}$ be the transitive collapse 
of both $N$ and $N'$ and let $\pi$ and $\pi'$ be the respective collapsing maps. 
Note that $\sigma=\pi ' \circ \pi^{-1}$ is the $\alpha$-isomorphism between $N$ and $N'$. 
Then $\sigma [N \land M] = N'\cap M$.  Pick also $P'\in N \land M$ which is $\alpha$-isomorphic to $P$. 
By \cref{trans} $P'$ and $\sigma (P')$ are also $\alpha$-isomorphic. 
Since $\sigma (P')\in M$, by the transitivity of $\cong_\alpha$ we get that $P$ is $\alpha$-isomorphic to  $\sigma(P')$.
This implies that $P\in_\alpha M$.

Now assume $P\in_\alpha N$ and $P\in_\alpha M$. By \cref{active-Magidor} we know that $\alpha \in N$.
Since $P$ is an $\alpha$-model, we conclude that $P\in N$.
If also $\alpha \in M$, we have that $N, P \in M$ and $N\land M = N\cap M$. Therefore, $P\in N \land M$. 
Assume now that $\alpha \notin M$ and let $\alpha^*= \min (M \cap \lambda \setminus \alpha)$.
Let $A= {\rm Hull}(M,V_\alpha)$. Since we assumed that $M$ is an $\alpha$-model, we have that $A\in \mathscr A_\alpha$
and $\alpha \in E_A$.
By \cref{alphasharpinE}   we also have that $\alpha^*\in E_A$. Fix $P^*, N^* \in M$ that are $\alpha$-isomorphic to $P$ and $N$ respectively.
By projecting them to $\alpha^*$ if necessary, we may assume $P^*, N^*\in \mathscr V_{\alpha^*}^A$.
Moreover, $N^*$ is a Magidor model from the point of view of $A$. 
Since $P^*\in_\alpha N^*$ and $\alpha^*$ is the least ordinal in $M$ above $\alpha$ we have 
\[
M\models \forall\delta\in E_A\cap \alpha^* P^*\in_\delta N^*.
\]
Moreover, $M\models "P^* \text{ is active at } \alpha^*"$. Since $\alpha^*$ is a limit point of $E_A$,
we can apply \cref{incont} in $A$ and conclude that  $\alpha^*\in N^*$ and $P^* \in N^*$.
Hence $P^*\in N^* \cap M$. Let $\sigma$ be the $\alpha$-isomorphism between $N^*$ and $N$.
Then $\sigma [N^* \cap M]= N\land M$. Hence $\sigma (P^*)\in N \land M$ and is $\alpha$-isomorphic to $P$.
It follows that $P\in_\alpha N \land M$.
\end{proof}

One feature of the meet  is the following absorption property.

\begin{proposition}\label{doublemeet}
Suppose $N\in \mathscr U$, $M\in \mathscr C$, and the meet $N\land M$ is defined. 
Let $\alpha \in E$, and suppose $P$ is a Magidor $\alpha$-model active at $\alpha$ such that $P\in_\alpha N \land M$. 
Then $P\land M = P \land (N \land M)$. 
\end{proposition}
\begin{proof}
Since $P\in_\alpha N \land M$ and $P$ is active at $\alpha$, so is $N\land M$, and hence both $N$ and $M$ are active at $\alpha$ as well. 
Let $\overline{P}$ be the transitive collapse of $P$.
Then $\overline{P}\in N\cap V_\kappa$, and since $N\cap V_\kappa$ is transitive, we have $\overline{P}\subseteq N$. 
Hence $\overline{P}\cap (N \land M)= \overline{P}\cap M$. It follows that $P\land M = P \land (N \land M)$.
\end{proof}

\begin{proposition}\label{meetin} Let $\alpha \in E$. Suppose $N\in \mathscr U$, $M\in \mathscr C$ and the meet $N\land M$ is defined
and active at $\alpha$. Suppose $P\in \mathscr V$ and $N,M\in_\alpha P$. Then $N\land M\in_\alpha P$.
\end{proposition}
\begin{proof} We may assume $M,N$ and $P$ are all $\alpha$-models. If $\alpha \in P$ then $N,M\in P$, and hence also $N\land M\in P$.
Suppose now $\alpha \notin P$.  Let $A={\rm Hull}(P,V_\alpha)$ and let  $\alpha^*=\min (P\cap {\rm ORD} \setminus \alpha)$. 
Note that $\alpha^*$ has uncountable cofinality in $A$. By \cref{alphasharpinE}  we have $\alpha^*\in E_A$. We can find $N^*,M^*\in P$ 
such that $N^*\cong_\alpha N$ and $M^*\cong_\alpha M$. We may assume that $N^*\in \mathscr U_{\alpha^*}^A$ and $M^*\in \mathscr C_{\alpha^*}^A$. 
Work for a moment in $A$. Since $N^*\in_\alpha M^*$, $\alpha^*$ is the least ordinal of $P$ above $\alpha$, and $N^*,M^*\in P$, we have
\[
A\models \forall \gamma \in E_A\cap \alpha^* N^* \in_\gamma M^*.
\]
By applying \cref{incontcount} inside $A$ we have that $N^*\! \in_{\alpha^*} \!  M^*$, and hence $A$ can compute the meet, say $Q$, of $N^*$ and $M^*$.
Then $Q\in P$, and by applying \cref{MPI} inside $A$, we get  $Q\rest \alpha = N^*\rest \alpha \land M^*\rest \alpha$. 
Hence $Q\cong_\alpha N\land M$.
\end{proof}

\begin{definition}
Let $\alpha \in E$ and let  $\mathcal M$ be a set of virtual models.
We let $\mathcal M\restriction\alpha=\{M\restriction\alpha:M\in\mathcal M\}$ and 
$\mathcal M^\alpha=\{M\restriction\alpha:~M\in \mathcal M \text{ is active at } \alpha\}$.
\end{definition}

We can now define what we mean by an  $\alpha$-chain.

\begin{definition}
Let $\alpha \in E$ and let $\mathcal M$ be a subset of $\mathscr U \cup \mathscr C$. 
We say $\mathcal M$ is an $\alpha$-{\em chain} if for all distinct  $M,N\in\mathcal M$, either $M\in_\alpha N$ or $N\in_\alpha M$,
or there is a $P\in\mathcal M$ such that either $M\in_\alpha P\in_\alpha N$  or $N\in_\alpha P\in_\alpha  M$.
\end{definition}

\begin{proposition}\label{finite-chain}
Suppose $\alpha \in E$ and $\mathcal M$ is a finite subset of  $\mathscr U \cup \mathscr C$.
Then $\mathcal M$ is an $\alpha$-chain if and only if there is an enumeration 
$\{ M_i : i <n \}$ of $\mathcal M$ such that $M_0\in_\alpha M_1\in_\alpha \cdots \in_\alpha M_{n-1}$. 
\end{proposition}

\begin{proof} Suppose first $\mathcal M$ is an $\alpha$-chain. Define the relation $<$ on $\mathcal M$ by letting $M < N$ iff $\kappa_M < \kappa_N$. 
It is straightforward to see that $<$ is a total ordering on $\mathcal M$. We can then let $\{ M_i: i <n \}$ be the $<$-increasing enumeration of $\mathcal M$. 
Conversely, suppose $\mathcal M = \{ M_i : i < n \}$ is the enumeration such that $M_0\in_\alpha M_1\in_\alpha \cdots \in_\alpha M_{n-1}$. 
Let $i<j <n$. If $j= i+1$ then $M_i\in_\alpha M_j$. Suppose $j > i+1$. If $M_j$ is a Magidor model or if there are no Magidor models between $M_i$ and $M_j$
by \cref{intermediate} we conclude that $M_i\in_\alpha M_j$. Otherwise let $k < j$ be the largest such that $M_k$ is a Magidor model. 
Then again by \cref{intermediate}, we conclude that $M_i\in_\alpha M_k\in_\alpha M_j$. 
\end{proof}

Let $\alpha \in E$ and let $\mathcal M$ be an $\alpha$-chain. Let $\in_\alpha^*$ be the transitive closure of $\in_\alpha$. 
Then $\in_\alpha^*$ is a total ordering on $\mathcal M$.
For $M,N\in\mathcal M$, we say $M$ is $\alpha$-{\em below} $N$ in $\mathcal M$, or equivalently $N$ is $\alpha$-{\em above}  $M$ in $\mathcal M$, if 
$M\in_\alpha^* N$ in $\mathcal M$. Now using the transitivity of $\in_\alpha^*$ we can form intervals in $\mathcal M$. 
Let 
\[
(M,N)^\alpha_{\mathcal M}=\{ P\in{\mathcal M} : M\in_\alpha^* P\in_\alpha^* N\}.
\]
Similarly we can define $[M,N]^\alpha_{\mathcal M}$, $[M,N)^\alpha_\mathcal M$, etc. 
For convenience we also allow that the endpoints of the intervals to be $\emptyset$ or $V_\lambda$; let $(\emptyset, N)^\alpha_{\mathcal M}$ 
be $\{P\in\mathcal M^\alpha :P\in_\alpha^*N\}$  in the first case, and let $(N,V_\lambda)_{\mathcal M}^\alpha$ be $\{P\in\mathcal M :~N\in^*_\alpha P\}$ in the second.

\section{Main Forcing}

We  fix an inaccessible cardinal $\kappa$ and a cardinal $\lambda > \kappa$ with 
$\cof(\lambda)\geq \kappa$ such that $(V_\lambda,\in,\kappa)$ is suitable. 
We start by defining the forcing notions $\mathbb M^\kappa_\alpha$, for all $\alpha \in E\cup \{ \lambda\}$. 

\begin{definition}\label{MF}
Suppose $\alpha\in E$. We say that $p=\M_p$ belongs to $\mathbb M^\kappa_\alpha$ if:
\begin{enumerate}
\item $\M_p$ is a finite subset of $\mathscr C_{\leq \alpha}\cup \mathscr U^\kappa_{\leq \alpha}$ that is closed under meets, 
\item $\M_p^\delta$ is a $\delta$-chain, for all $\delta \in  E\cap (\alpha +1)$.
\end{enumerate}
We let $\M_q\leq \M_p$ if for all $M\in \M_p$ there is $N\in \M_q$ such that  $N \! \restriction \! \eta(M)=M$.
Finally, let $\mathbb M^{\kappa}_\lambda=\bigcup \{ \mathbb M^{\kappa}_\alpha: \alpha \in E\}$ with the same ordering.
\end{definition}

\begin{remark} Conditions (1) and (2) can be merged to a single condition. Let us say that a $\delta$-chain $\M$ consisting of models active at $\delta$
is {\em closed under meets} if for every $M,N\in \M$, if the meet $M\land N$ is defined and active at $\delta$ then $M\land N\in \M$. 
Thus we can simply say that $\M_p^\delta$ is a $\delta$-chain closed under meets, for all $\delta \in E\cap (\alpha +1)$. 
The order is natural since if $N\rest \eta(M)= M$, then $N$ carries all the information that $M$ does. 
\end{remark}

If $\kappa < \lambda$ are supercompact cardinals then $\mathbb M^\kappa_\lambda$ forces the principles ${\rm GM}^+(\omega_3,\omega_1)$.
Let us explain what happens. Suppose $G$ is generic over $\mathbb M^\kappa_\lambda$. 
Then $\omega_1$ is preserved, but $\kappa$ becomes $\omega_2$ and $\lambda$ becomes $\omega_3$ in $V[G]$.
Let  $\mathcal M_G= \bigcup G$, and let $G_\alpha = G\cap \mathbb M^\kappa_\alpha$, for $\alpha \in E$. 
One can show that for every  $\alpha \in E$ and $\beta > \alpha$, $V_{\beta}[G_\alpha]$ is a strong $\omega_1$-guessing model in $V[G]$. 
To see this  fix some $\delta \in E \setminus \beta$ with ${\rm cof}(\delta)<\kappa$. 
One shows that if $M$ is a Magidor model in $\M_G^\delta$ then $M[G_\alpha]$ is an $\omega_1$-guessing model in $V[G]$. 
Moreover, if $M$ is a Magidor model which is a limit of Magidor models in the $\delta$-chain $\M_G^\delta$ then 
$M\cap V_\delta$ is covered by the union of the previous models in $\M_G^\delta$.
Therefore, if we let $\mathcal G = \{ (M\cap V_\beta)[G_\alpha]: M \in \M_G^\delta \cap \mathcal U^\kappa_\delta \}$,
then $\mathcal G$ is an increasing sequence of $\omega_1$-guessing models which is continuous at uncountable limits 
and the union of this sequence is $V_\beta[G_\alpha]$. 
We will actually present a proof not for the forcing $\mathbb M^\kappa_\lambda$, but for a slight variation 
$\mathbb P^\kappa_\lambda$. We would like to arrange that in addition the set 
$\{ \sup (M\cap \kappa) : M\in \M_G^\delta\}$ be a club in $\kappa$, for all $\delta \in E$ with ${\rm cof}(\delta)<\kappa$. 
In order to achieve this we will add {\em decorations} to the conditions of $\mathbb M^\kappa_\lambda$. 
This device,  introduced by Neeman  \cite{NE2014}, consists of attaching to each model $M$ of an $\in$-chain a finite set $d_p(M)$
which belongs to all models $N$ of the chain such that $M\in N$. In a stronger condition this finite set is allowed to increase.
The main point is that $d_p(M)$ controls what models can be added $\in$-above $M$ in stronger conditions.
In our situation there are some  complications. First, we have not one chain, but a $\delta$-chain, for each $\delta \in E$. 
It is therefore reasonable to have decorations for each level $\delta \in E$. Now, models from a higher level project to lower levels
at which they are active, but also in order to arrange strong properness for countable models, some models from lower levels
will be {\em lifted} to higher levels and put on the chain. This imposes a subtle interplay between the decorations on different levels. 
In order to describe this precisely, we need to make some preliminary definitions. 

\begin{definition}\label{L(p)} Suppose $\M_p\in \mathbb M^\kappa_\lambda$. 
Let $\mathcal L(\M_p)= \{ M\restriction \alpha: M\in \M_p \text{ and } \alpha \in a(M)\}$. 
\end{definition}

\begin{definition}\label{free} Suppose $\M_p\in \mathbb M^\kappa_\lambda$. We say that $M\in \mathcal L(\M_p)$ is $\M_p$-{\em free}
if  every $N\in \M_p$ with $M\in_{\eta(M)}N$  is strongly active at $\eta(M)$. 
Let $\mathcal F(\M_p)$ denote the set of all $M\in \mathcal L(\M_p)$ that are $\M_p$-free. 
\end{definition}
 
Note that if $\M_q\leq \M_p$ then $\mathcal L(\M_p)\subseteq \mathcal L(\M_q)$ and $\mathcal F (\M_q)\cap \mathcal L(\M_p)\subseteq \mathcal F(\M_p)$. 
In other words, a node $M\in \mathcal L(\M_p)$ that is not  $\M_p$-free is not $\M_q$-free, for any  $\mathcal M_q\leq \mathcal M_p$.
We are now ready to define our main forcing notion. 
 
\begin{definition}\label{PF} Suppose $\alpha\in E\cup \{ \lambda\}$. We say that a pair $p=(\M_p,d_p)$ belongs to $\mathbb P^\kappa_\alpha$ if 
 $\M_p\in \mathbb M^\kappa_\alpha$, $d_p$ is a finite partial function from $\mathcal F(\M_p)$ to ${\mathcal P}_{\omega}(V_\kappa)$, and 
\medskip
\begin{itemize}
\item[$(*)$]  \hspace{5mm} if $M \in \dom (d_p)$, $N\in \M_p$, and $M\in_{\eta(M)} N$, then $d_p(M)\in N$. 
\end{itemize}
\medskip
We say that $q\leq p$ if $\M_q \leq \M_p$, and for every $M\in \dom(d_p)$ there is $\gamma \in E\cap (\eta(M)+1)$
such that $M\restriction \gamma \in \dom(d_q)$ and $d_p(M)\subseteq d_q(M\restriction \gamma)$. 
\end{definition}

\begin{remark} We refer to $d_p$ as the decoration of $p$. The point is that if $M\in \dom (d_p)$ is a $\delta$-model then 
$d_p(M)$ constraints what models $N$ with $M\in_\delta N$ can be put on  $\M_q^\delta$, for any $q\leq p$.
In general, $M$ may not be $\M_q$-free, in which case $M\notin \dom(d_q)$, but then we have some $\gamma \leq \delta$ 
 such that $M\restriction \gamma$ is $\M_q$-free and $d_p(M)\subseteq d_q(M\restriction \gamma)$. 
 Note that then we must have $d_p(M)\in N$, for any $N\in \M_q$ such that $M\in_\delta N$.

The ordering on $\mathbb P^\kappa_\lambda$ is clearly transitive. We will say that $q$ is {\em stronger} than $p$
if $q$ forces that $p$ belongs to the generic filter, in order words, any $r\leq q$ is compatible with $p$.
We write $p\sim q$ if each of $p$ and $q$ is stronger than the other. We identify equivalent conditions, often without saying it. 
Our forcing does not have meets, but if $p$ and $q$ do have a weakest lower bound we will denote it by $p\land q$. 
To be precise we should refer to $p\land q$ as the $\sim$-equivalence class of a weakest lower bound, 
but we ignore this point since it should not cause any confusion. 
Note that if $p\in \mathbb P^\kappa_\alpha$ and $M\in \M_p$ is a $\delta$-model that is not active at $\delta$,
we may replace $M$ by $M\restriction \alpha(M)$ and we get an equivalent condition. 
Thus, if $\alpha \in E$ and $\cof(\alpha)\geq \kappa$, then $\mathbb P^\kappa_\alpha$ is forcing equivalent to 
$\bigcup \{ \mathbb P^\kappa_\gamma : \gamma \in E\cap \alpha\}$. 
\end{remark}

\begin{convention} Suppose $p\in \mathbb P^\kappa_\lambda$ and $\delta \in E$. If $M,N\in \M_p^\delta$ with $M\in_\delta^* N$,
we will write $(M,N)_p^\delta$ for the interval $(M,N)_{\M_p}^\delta$, and similarly, for $[M,N)_p^\delta$, $(M,N]_p^\delta$, etc. 
\end{convention}

Suppose $\alpha, \beta \in E$ and  $\alpha \leq \beta$. For  every $p\in\mathbb P^{\kappa}_\beta$, 
we let $\M_{p\restriction \alpha} = \M_p\rest \alpha$ and $d_{p\restriction \alpha} =d_p\rest \mathcal F(\M_p\rest \alpha)$. 
It is easily seen that $p= (\M_{p\restriction \alpha},d_{p \restriction \alpha})\in \mathbb P^{\kappa}_\alpha$. 
The following  is straightforward.

\begin{lemma}\label{CS}
Suppose $\alpha,\beta \in E$ with $\alpha \leq \beta$. Let $p\in\mathbb P^{\kappa}_\beta$ 
and let $q\in\mathbb P^{\kappa}_\alpha$ be such that $q \leq p\restriction \alpha$.
Then there exists $r\in\mathbb P^{\kappa}_\beta$ such that $r\leq p,q$.
\end{lemma}

\begin{proof} We let $\M_r = \M_p \cup \M_q$. Note that $\M_r$ is closed under meets. 
 We define $d_r$  by letting  $d_r(M)=d_q(M)$ if $M\in \dom(d_q)$, and $d_r(M)= d_p(M)$ if $M\in \dom(d_p)$ with $\eta(M) > \alpha$. 
It is straightforward that $r$ is as required. 
\end{proof}

\begin{remark} The condition $r$ from the previous lemma is the greatest lower bound of $p$ and $q$, so we will write $r = p \land q$. 
\end{remark}

\begin{corollary}\label{alpha-complete-suborder} 
Suppose $\alpha \in E$. Then $\mathbb P^\kappa_\alpha$ is a complete suborder of $\mathbb P^\kappa_\lambda$. 
\end{corollary}

\begin{proof}[\nopunct] 
 \end{proof}

Our goal is to prove that our poset $\mathbb P^\kappa_\lambda$ is strongly proper for an appropriate class of models. 
We start by showing that if a condition $p$ belongs to a model $M$ we can always add $M$ to $\M_p$ and form a new condition. 

\begin{lemma}\label{topM} Let $p\in \mathbb P^{\kappa}_\lambda$ and $M\in \mathscr C \cup \mathscr U$ be such that $p\in M$.
Then there is a weakest condition $p^M\leq p$ with $M\in \M_{p^M}$.
\end{lemma}
\begin{proof}
Suppose first that $M$ is a Magidor model. Then we let $\mathcal M_{p^M}= \M_p$
and $d_{p^M}=d_p$. It is straightforward that $p^M=(\M_{p^M},d_{p^M})$ is as required. 

 Now assume that $M$ is countable. We let $\M_{p^M}$ be the closure of $\M_p\cup \{ M\}$ under meets. 
  Fix $\delta\in E$. We  show that $\mathcal M_p^\delta$ is an $\in_\delta$-chain. 
 We may assume that  $M$ is active at $\delta$ since otherwise $\mathcal M_{p^M}^\delta=\emptyset$.
 By \cref{doublemeet} we know that the only models added to $\mathcal M_p^\delta$ in order to form $\mathcal M_{p^M}^\delta$ are $M\restriction \delta$ 
 and $N\land M$ for $N\in \mathcal M_p^\delta$ such that $N\land M$ is active at $\delta$.  
 Suppose $N\in \M_p^\delta$ is such a model, and let $P$ be the $\in_\delta$-predecessor of $N$ in $\M_p^\delta$, if it exists. 
First note that $N\cap M\in N$ since $N$ is closed under countable sequence. 
Therefore, $N\land M\in N$. Moreover, if $P$ exists by \cref{meetactiveness} we have that $P\in_\delta N\land M$. 
This establishes that $\M_{p^M}^\delta$ is a $\delta$-chain. 
 
 Let us now define the decoration $d_{p^M}$. Suppose $N\in \dom(d_p)$ is a $\delta$-model. Then $\delta\in M$. 
 If $M$ is strongly active at $\delta$, then by \cref{meet-trace}, for every Magidor model $P\in \M_p$
 if  $P\land M$ is  active at $\delta$ then it is strongly active at $\delta$. Hence $N$ is $\M_{p^M}$-free.
 We  then keep $N$ in $\dom(p^M)$ and let $d_{p^M}(N)=d_p(N)$. 
 Now, suppose $M$ is not strongly active at $\delta$. This means that $\delta$ has uncountable cofinality in $M$.
 Let $\bar\delta = \sup (M\cap \delta)$ and note that $\bar\delta$ is a limit point of $E$. 
 We claim that $N\rest \bar\delta$ is $\M_{p^M}$-free. Indeed, if there is $P\in \M_{p^M}$ such that 
 $N\in_{\bar\delta} P$ and $P$ is  not strongly active at $\bar\delta$, then $P\in M$, and hence $\eta(P)\geq \delta$. 
 Moreover, $P$ is active but not strongly active at $\delta$ as well. 
 Since $N\in_{\bar\delta}P$ and $N,P\in M$ it follows that $N\in_{\gamma} P$, for unboundedly many $\gamma \in E\cap \delta\cap M$. 
 But then by \cref{incontcount} applied in $M$ we conclude that $N\in_\delta P$, and hence $N$ is not $\M_p$-free, a contradiction. 
 Notice also that if $P\in \M_p$ and $N\in_{\bar\delta} P$ then by  \cref{incontcount} again we must have
 that $N\in_\delta P$ and thus $d_p(N)\in P$. Therefore, we can replace $N$ by $N\rest \bar\delta$ and 
 let $d_{p^M}(N\rest \bar\delta)=d_p(N)$. 
 It is straightforward to check that $p^M$ is a weakest extension of $p$ such that $M\in \M_{p^M}$.
\end{proof}

If $N,M$ are virtual models it will be convenient to set $\alpha (N,M)= \max (a(N)\cap a(M))$. 

\begin{definition}\label{restriction-magidor} Suppose $p\in \mathbb P^\kappa_\lambda$ and $M\in \mathcal L(\M_p)$ is a Magidor model. 
For $N\in \M_p$ we let $N\restriction M = N \restriction \alpha (N,M)$ if $\kappa_N < \kappa_M$, 
otherwise $N\restriction M$ is undefined. 
Let  
\[
\M_{p \restriction M} = \{ N \rest M : N \in \M_p\}.
\]
Let $d_{p \restriction M}= d_p \rest (\dom(d_p)\cap M)$, and let $p \restriction M = (\M_{p \restriction M}, d_{p \restriction M})$.
\end{definition}

\begin{lemma}\label{rest-magidor-prop} Suppose $p\in \mathbb P^\kappa_\lambda$ and $M\in \mathcal L(\M_p)$ is a Magidor model. 
Then $p\restriction M \in \mathbb P^\kappa_\lambda \cap M$ and $p\leq p\restriction M$. 
\end{lemma}

\begin{proof}
Since $p$ is a condition, we have that if $N\in \M_p$ and $\kappa_N < \kappa_M$, then $N\in_\gamma^* M$, for 
all $\gamma \in a(N)\cap a(M)$. By \cref{intermediate} we then conclude that $N\in_\gamma M$, for all such  $\gamma$.
By \cref{active-Magidor} we have that $\alpha(N,M)\in M$, and hence $N \restriction \alpha(N,M)\in M$. 
We also have that $d_{p \restriction M}\in M$, thus $p \restriction M \in M$. 
Let us check that $\M_{p\restriction M} \in \mathbb M^\kappa_\lambda$. Suppose $\delta \in E$.
If $M$ is not active at $\delta$ then  $\M_{p\restriction M}^\delta$ is empty, otherwise it is equal to 
$(\emptyset,M\! \restriction \! \delta)_{p}^\delta$, which is obviously  a $\delta$-chain.
To check that $\M_{p\restriction M}$ is closed under meets, suppose $N\restriction M,P\restriction M\in \M_{p\restriction M}$ 
and their meet is defined. 
Note that then $N \land P$ is also defined and, by \cref{MPI} $(N\land P) \restriction M = N\restriction M \land P\restriction M$. 
It is straightforward to check that every $N\in \dom(d_{p\restriction M})$ is $\M_{p\restriction M}$-free, 
and $(*)$ from \cref{PF} holds. 
Finally, the fact that $p\leq p\restriction M$ follows from the definition. 
\end{proof}

\begin{lemma}\label{Magidor-compatible}
Suppose $p\in\mathbb P^\kappa_\lambda$ and $M\in \mathcal L(\mathcal M_p)$ is a Magidor model. 
Suppose $q\in M\cap\mathbb P^\kappa_\lambda$ extends $p\restriction M$.
Then $q$ is compatible with $p$ and the meet $p\land q$  exists. 
\end{lemma}
\begin{proof}
We define $r\in \mathbb P^\kappa_\lambda$ and check that it is a weakest condition extending $p$ and $q$. 
Let $\mathcal M_r=\mathcal M_p\cup\mathcal M_q$. 
We check that if $\delta \in E$, then $\M_r^\delta$ is a $\delta$-chain closed under meets, meaning
if $P,Q\in \M_r^\delta$ and the meet $P\land Q$ is defined and active at $\delta$ then $P\land Q \in \M_r^\delta$. 
Fix such $\delta \in E$. If $M$ is not active at $\delta$, then $\M_r^\delta = \M_p^\delta$ and thus has
the required property since $p$ is a condition. Now, suppose $M$ is active at $\delta$. 
If $R \in \M_r^\delta$ and $R\in_\delta^*M$, then by \cref{intermediate} we know that $R \in_\delta M$, 
and by \cref{active-Magidor} we get that $\delta \in M$. Hence $R\in M$ and therefore $R\in \M_q^\delta$. 
Therefore, $\M_r^\delta$ is the union of $\M_q^\delta$ and $[M\restriction \delta, V_\lambda)_{p}^\delta$,
and hence is a $\delta$-chain. Now suppose $P,Q\in \M_r^\delta$ and their meet is defined
and active at $\delta$. We need to check that $Q\land P \in \M_r^\delta$. 
If both $P$ and $Q$ belong either to $\M_q^\delta$ or $\M_p^\delta$, this follows from the fact
that $p$ and $q$ are conditions.  Since $Q\in_\delta P$ and $\M_q^\delta$ is an $\in_\delta^*$-initial segment of $\M_r^\delta$,
we may assume $Q\in \M_q^\delta$ and $P\in \M_p^\delta\setminus \M_q^\delta$. 
The proof goes by induction on the number of Magidor models on the $\delta$-chain 
$[M\restriction \delta, P)_{p}^\delta$. 
If $M\in_\delta P$ then $M\land P \in \M_p^\delta$ and is $\delta$-below $M\restriction \delta$, hence
belongs to $\M_q^\delta$. On the other hand, by  \cref{doublemeet} we have 
$Q\land P= Q \land (M \land P)$, and since $\M_q^\delta$ is closed under meets we get 
that $Q\land P\in \M_q^\delta$. 
In general, if $N$ is the $\in_\delta^*$-largest Magidor model in $[M\restriction \delta, P)_{p}^\delta$,
by \cref{intermediate}, we have that $Q\in_\delta N \in_\delta P$. In particular, $N\land P$ is defined
and by \cref{doublemeet} we have that $Q \land P = Q \land (N\land P)$. 
Now, we are done if $N\land P \in \M_q^\delta$ as $q$ is a condition.
Otherwise, it belongs to the interval $[M\restriction \delta, P)_{p}^\delta$.
Then there are fewer Magidor models in $[M\rest \delta, N\land P)^\delta_p$ and thus we can use
the induction hypothesis. 

Let $d_r = d_q \cup d_p \rest (\dom(d_p)\setminus M)$. 
Let us check that every $N\in \dom(d_r)$ is $\M_r$-free. For simplicity, let $\eta=\eta (N)$. 
If $N\in \dom(d_p)\setminus M$, then 
there is no $P\in \M_q$ such that $N\in_\eta P$, and hence the conclusion follows from the fact that $p$ is a condition.
Suppose now $N\in \dom(d_q)$ and $P\in \M_r$ is such that $N\in_\eta P$. We have to check 
that $P$ is strongly active at $\eta$. We may assume that $P$ is a countable model. 
If $P\rest \eta$ is $\eta$-below $M$, then $P\rest M$  is defined and $P\rest M\cong_{\eta} P$, 
therefore, the conclusion follows from the fact that $q$ is a condition. 
If $P\rest \eta$ is $\eta$-above $M$, then $M\land P$ is defined and belongs to $\M_p$.
Moreover,  by \cref{meetactiveness}, $N\in_{\eta} M\land P$. 
Now $(M\land P)\rest M$ is defined, and belongs to $\mathcal \M_q^{\alpha(M\land P,M)}$,  and is strongly active at $\eta$ since $N$ is $\M_q$-free. 
Therefore, $P$ is also strongly active at $\eta$. 
The fact that $d_r$ satisfies condition $(*$) from \cref{PF} is straightforward. 
Finally, the fact that $r$ is the weakest common extension of $p$ and $q$ follows readily from the definition. 
 \end{proof}

By \cref{topM} and  \cref{Magidor-compatible} we immediately get the following. 

 \begin{theorem}\label{Magidor-strongly-proper}
   The forcing $\mathbb P^\kappa_\lambda$ is $\mathscr U$-strongly proper.
   \end{theorem}
\begin{proof}[\nopunct] 
 \end{proof}
  
We now proceed to define an analogue of $p\rest M$ for  countable models $M\in \mathcal L(\M_p)$.
The situation here is more subtle since $p\rest M$ may not belong to the original forcing, only its version as defined in $M$.
We first analyze the part  involving $\M_p$. It will be useful to make the following definition. 

\begin{definition}\label{countable-restriction} Let $\M$ be a subset of $\mathscr C \cup \mathscr U$ and $M\in \mathscr C$. 
For $\delta \in E$, we let $(\M \restriction M)^\delta = \{ N \in \M^\delta : N \in_\delta M\}$. 
\end{definition}

\begin{lemma}\label{gapmeet}
Let  $\M_p\in\mathbb M^\kappa_\lambda$ and $\delta \in E$. Suppose $M\in\mathcal M_p^\delta$ is  countable. 
Then $(\M_p\rest M)^\delta$ is a $\delta$-chain closed under meets and
\[
(\M_p\restriction M)^\delta = (\emptyset,M\restriction \delta)_p^\delta \setminus 
\bigcup\{ [N\land M,N)^\delta_{p}:N\in (\mathcal M_p\restriction M)^\delta  \text{ and is a Magidor model}\}.
\]
Here, if $N\land M$ is defined and not active at $\delta$, by $[N\land M,N)_p^\delta$ we mean $(\emptyset, N)_p^\delta$. 
\end{lemma}

\begin{proof}
It is clear that  $(\mathcal M_p\restriction M)^\delta\subseteq (\emptyset,M)_p^\delta$.
Suppose $P\in \M_p^\delta$ and $P\in_\delta M$. 
Then, for any Magidor model $N\in (P,M)_p^\delta$, we have $P\in_\delta N$ by \cref{intermediate}.
Then by   \cref{meetactiveness} we have that $P\in_\delta N\land M$. 
Conversely, suppose $P$ is in  $(\emptyset,M)^\delta_{p}$, but not in $(\mathcal M_p\restriction M)^\delta$.
Then, by \cref{intermediate} again, there must be a Magidor model $N\in \M_p^\delta$ such that $P\in_\delta N \in_\delta M$. 
Let $N$ be the $\in_\delta^*$-least such model. If $N\land M$ is not active at $\delta$, then $P\in (\emptyset,N)_p^\delta$.
Suppose $N\land M$ is active at $\delta$. We have to show that either $P=N\land M$ or $N\land M \in_\delta^*P$. 
Indeed, otherwise we have $P\in_\delta^* N\land M$. Note that there cannot be a Magidor model $Q\in \M_p^\delta$ 
with $P\in_\delta Q \in_\delta N\land M$ since then we would have $Q\in_\delta M$ as well, and this contradicts the minimality of $N$.
Since $\M_p^\delta$ is a $\delta$-chain, by \cref{intermediate} we conclude that $P\in_\delta N\land M$, but then also $P\in_\delta M$, 
a contradiction. The fact that $(\M_p\rest M)^\delta$ is a $\delta$-chain follows from the above analysis. 
By \cref{meetin} it is also closed under meets. 
\end{proof}

\begin{lemma}\label{full-countable} Let $p\in \mathbb P^\kappa_\lambda$ and $M,N\in \M_p$.
If there is $\gamma \in a(M)\cap a(N)$ such that $N\in_\gamma M$, then $N\in_\delta M$, for all $\delta \in a(M)\cap a(N)$. 
\end{lemma}

\begin{proof} Let $\alpha = \alpha(M,N)$. If $N\in_\alpha M$ then $N\in_\gamma M$, for all $\gamma \in a(M)\cap a(N)$, by 
\cref{projection-membership}. Now suppose $N\notin_\alpha M$. If $M$ is a Magidor model we have $\kappa_M\leq \kappa_N$,
and hence there is no $\gamma$ such that $N\in_\gamma M$. 
Assume now that $M$ is countable. Then by \cref{gapmeet} there is a Magidor model $P\in \M_p^\delta$ with $P\in_\delta M$
such that $N\restriction \delta$ is in the interval $[P\land M, P)_p^\delta$. 
Now, $P\land M$ is active at all $\gamma \in a(N)\cap \alpha$ and $N\restriction \gamma$ is in the interval 
$[(P\land M)\restriction \gamma, P\restriction \gamma)_p^\gamma$,
for all such $\gamma$. But then, by \cref{gapmeet} again, $N \notin_\gamma M$, for all $\gamma \in a(M)\cap a(N)$.
\end{proof}

It would be useful to introduce some notation. 

\begin{definition} Suppose $M\in \mathscr V$ and $\M$ is a finite subset of $\mathscr V$. Let $\alpha \in E$. 
We write $\M \in_\alpha M$ if $N\in_\alpha M$, for all $N\in \M$. 
\end{definition}

\begin{lemma}\label{level-merging} Suppose $\M_p\in \mathbb M^\kappa_\lambda$ and $\delta \in E$. 
Suppose $M\in \M_p^\delta$ is a countable model,  $\M\in_\delta M$ is a finite $\delta$-chain closed under meets,
and  $(\M_p\rest M)^\delta \subseteq \M$. Then the closure of $\M_p^\delta \cup \M$ under meets that are active at $\delta$
is a $\delta$-chain. 
\end{lemma}

\begin{proof} 
 Let us first show that $\M_p^\delta \cup \M$ is a $\delta$-chain. Indeed, by \cref{gapmeet} 
 it is obtained by adding  to $\M$ the intervals $[N\land M, N)_p^\delta$, where $N\in (\M_p\rest M)^\delta$ is a Magidor 
 model, and the interval $[M,V_\lambda)_p^\delta$. Consider one such interval, say $[N\land M,N)_p^\delta$. 
 If $P$ is the last model of $\M$ before $N$ then $P\in_\delta M$ by the assumption that $\M\in_\delta M$,  
 and $P\in_\delta N$ by \cref{intermediate}. Hence by \cref{MPI} we have that $P\in_\delta N\land M$. 
 It follows that $\M_p^\delta \cup \M$ is a $\delta$-chain. 
 
 Let us now consider what happens when we close $\M_p^\delta \cup \M$ under meets that are active at $\delta$. 
Suppose $Q, P\in \M_p^\delta \cup \M$, $Q$ is a Magidor model, $P$ is countable, and $Q\in_\delta P$.
If $P\in \M$ then $Q\in_\delta P\in_\delta M$, and hence by \cref{intermediate} $Q\in M$, and so $Q\in \M$ as well. 
Since $\M$ is closed under meets, we have that $Q\land P\in \M$. 
Now suppose $P\in \M_p^\delta \setminus \M$. By \cref{gapmeet} we have that $P\in [M,V_\lambda)_p^\delta$
or $P\in [N\land M,N)_p^\delta$, for some Magidor model $N\in (\M_p\rest M)^\delta$.
The two cases are only notationally different, so let us assume that there is  a Magidor model  $N\in (\M_p\rest M)^\delta$ 
such that $ P\in [N\land M,N)_p^\delta$. 
We may assume that $Q\in \M$.
Note that  $Q\in_\delta N$ and $Q\in_\delta M$, hence by \cref{meetactiveness} $Q\in_\delta N\land M$.
If there is a Magidor model $R\in \M_p$ such that $Q\in_\delta^* R\in_\delta^* P$, 
let $R$ be the $\in_\delta^*$-largest such model. By \cref{intermediate} we have that $R\in_\delta P$ and
hence $R\land P$ is defined and is below $P$ on the $\delta$-chain $\M_p^\delta$. 
Moreover, since $\M_p\cup \M$ is a $\delta$-chain, also by \cref{intermediate}, we have $Q\in_\delta R$. 
Now, by \cref{meetactiveness} we have that $Q\in_\delta R \land P$, and by \cref{doublemeet} we have  $Q\land P = Q\land (R\land P)$. 
Therefore, we may assume that there are no Magidor models $R\in \M_p^\delta$ with $Q\in_\delta R \in_\delta P$. 
Now, let $\{ P_i : i < k\}$ list all countable models on the chain $[N\land M,N)_p^\delta$ below the first Magidor model, if it exists. 
Then $P_0= N\land M$ and $P= P_j$, for some $j$. Note that $Q\in_\delta P_i$, for all $i<k$, again by \cref{intermediate}.
Now let $S$ be the $\in_\delta^*$-predecessor of $N$  on the $\delta$-chain $\M_p^\delta \cup \M$, if it exists, otherwise let $S$ be $\emptyset$. 
Note that $S\in_\delta N\land M$. Indeed, if $S\in \M_p$ this follows from \cref{intermediate}, and the fact that 
there are no Magidor models in $(S,N\land M)_p^\delta$.  If $S\in \M$ then $S\in_\delta M$ and thus $S\in_\delta N\land M$. 
Now, by \cref{intermediate} we have that $S\in_\delta P_i$, for all $i< k$. 
Hence, by \cref{meetactiveness} we have $S\in_\delta Q\land P_i$, for all $i$. 
By \cref{meetin} we have $Q\land P_i\in_\delta P_{i+1}$, for all $i < k-1$. 
Since $Q$ is a Magidor model, we also have that $Q\land P_i \in_\delta Q$, for all $i<k$.
By \cref{meetactiveness} again, we have that $Q\land P_i \in_\delta Q\land P_{i+1}$, for all $i< k-1$. 
Therefore, $S\in_\delta Q\land P_0\in_\delta \ldots \in_\delta Q\land P_{k-1}\in_\delta Q$,
and $Q\land P$ appears on this chain. If $S=\emptyset$ then an initial segment of this chain may be nonactive
at $\delta$, but the remainder is still a $\delta$-chain. 
 \end{proof}

Now, suppose $p\in \mathbb P^\kappa_\lambda$ and $M\in \M_p$ is a countable $\beta$-model, for some $\beta \in E$. 
Let $A={\rm Hull}(M,V_{\beta})$. Then $A\in \mathscr A_\beta$. 
Note that $E_A\cap \beta= E\cap \beta$, and if $\beta\in A$ then $\beta\in E_A$. 
Also, note that the definitions of $\mathbb P^\kappa_\alpha$ and the order  relation is $\Sigma_1$ with parameter $V_\alpha$. 
For $\alpha \in E_A$, let $(\mathbb P^\kappa_\alpha)^A$ be the version of $\mathbb P^\kappa_\alpha$ as defined in $A$.  
Then  $(\mathbb P^\kappa_\alpha)^A= \mathbb P^\kappa_\alpha$ if $\alpha < \beta$, and 
$(\mathbb P^\kappa_\beta)^A\subseteq \mathbb P^\kappa_\beta$. 
We will let $\mathscr V^M_\alpha= \mathscr V^A_\alpha \cap M$, and
$(\mathbb P^\kappa_\alpha)^M= (\mathbb P^ \kappa_\alpha)^A\cap M$, if $\alpha \in E_A\cap M$. 
Suppose $N\in \M_p$ and  $N\in_\delta M$, for some $\delta \in a(M)\cap a(N)$,
Then by \cref{full-countable}, $N\in_\alpha M$, where $\alpha = \alpha (M,N)$. 
Note that if $M$ is a standard $\beta$-model then $\alpha < \beta$. 
It may be that $\alpha \notin M$, but then, if we let $\alpha^*= \min (M\cap {\rm ORD}\setminus \alpha)$, we have
that $\alpha^*\in E_A\cap M$, and $\alpha^*$ is of uncountable cofinality in $A$. 
By the previous remarks, if $M$ is a standard $\beta$-model or $\beta\in M$ then $\alpha^*\in E\cap (\beta +1)$,
otherwise $\alpha^*$ may be in the nonstandard part of $M$.
Since $N\in_\alpha M$, there is a model $N^*\in M$ with $N^*\in \mathscr V^A$
which is $\alpha$-isomorphic to $N$. Now, $M$ can compute $N^*\restriction \alpha^*$, hence we may assume
$N^*\in \mathscr V^A_{\alpha^*}$. Moreover, such $N^*$ is unique. Indeed, if there is another model $N^{**}\in M$ with the same
property, since $\alpha^*$ is the least ordinal in $M$ above $\alpha$ and $N^*\cong_\alpha N^{**}$
we would have that $N^*\cong_\delta N^{**}$, for all $\delta \in E_A \cap \alpha^*\cap M$.
Hence, by \cref{isocont} applied in $M$,  we would have that $N^*=N^{**}$.
This justifies the following definition.

\begin{definition}\label{restriction-countable} 
Suppose $p\in \mathbb P^\kappa_\lambda$ and let $M\in \mathcal L(\M_p)$ be a countable $\beta$-model, for some $\beta\in E$. 
Let $N\in \M_p$ and let $\alpha = \alpha (M,N)$. If $N\in_\alpha M$ we let $\alpha^*= \min (M\cap {\rm ORD}\setminus \alpha)$.
We define $N\restriction M$ to be the unique  $N^*\in \mathscr V^M_{\alpha^*}$ such that $N^* \cong_\alpha N$.  
Otherwise we leave $N\restriction M$ undefined. Let
 \[
 \M_ {p\restriction M} = \{ N\restriction M : N\in \M_p\}, \text{ and}
 \]
\[
\dom(d_{p \restriction M})= \{ N\restriction M: N\in \dom(d_p) \text{ and } N\in_{\eta (N)}M \}.
\]
If $N\in \dom(d_p)$ and $N\in_{\eta(N)} M$, let $d_{p\restriction M}(N\rest M)= d_p(N)$.
Let $p\rest M = (\M_{p\restriction M},d_{p \restriction M})$. 
\end{definition}

\begin{remark} Suppose $N\in \dom (d_p)$ and let $\eta= \eta(N)$. If $N\in_\eta M$ then $M$ is strongly active at $\eta$ since $N$ is $\M_p$-free. 
If $\eta \in M$ then we put $N$ in $\dom(d_{p\restriction M})$ and keep the same decoration at $N$. 
If $\eta  \notin M$ we lift $N$ to the least level $\eta^*$ of $M$ above $\eta$, we put the resulting model $N^*$ in $\dom(d_{p\restriction M})$
and copy the decoration of $N$ to $N^*$.  If $P\in \M_p$ is such that $P\rest \eta =N$ then $(P\rest M)\rest \eta^* = N^*$.
Moreover, from $N^*$ we can recover $N$ as $N^*\rest \sup (\eta^*\cap M)$. Thus, the function $d_{p\restriction M}$ is well defined. 
Note also that $p\rest M \in M$. 
\end{remark}

\begin{proposition}\label{countable-condition}
Suppose $p\in\mathbb P^\kappa_\lambda$ and $M\in \mathcal L(\M_p)$ is a countable $\beta$-model, for some $\beta \in E$.
Let 
\[\alpha= \max \{ \alpha(N,M): N\in_{\alpha(N,M)}M \text{ and } N\in \M_p\}.
\]
Let $\alpha^*=\min (M\cap {\rm ORD}\setminus \alpha)$. Then $p\restriction M\in (\mathbb P^\kappa_{\alpha^*})^M$.
\end{proposition}

\begin{proof} Let $A={\rm Hull}(M,V_\beta)$ and work in $A$.  It is clear that $\M_{p\restriction M}$ is a finite subset of 
$\mathscr C_{\leq \alpha^*}^A \cup \mathscr U_{\leq \alpha^*}^A$. 
We first show that  $\M_{p\restriction M}^\gamma$ is a $\gamma$-chain closed under meets, for all $\gamma \in E_A\cap (\alpha^*+1)$. 
Fix such $\gamma$ and let $\delta=\min (M\cap {\rm ORD}\setminus \gamma)$ and $\bar\delta = \sup (M\cap \delta)$. 
If $\bar\delta = \delta$ then $\gamma = \delta$, and hence $\gamma \in M$ and the conclusion follows from the fact that $p$ is a condition
and \cref{gapmeet}.
Let us assume now that $\bar\delta < \delta$. Note that then $\bar\delta,\delta \in E$, $\delta$ is of uncountable cofinality in $M$,
and is a limit point of $E$. 
Note that if $P\in \M_{p\restriction M}$ is a $\delta$-model that is active at $\gamma$ then $a(P)$ is cofinal in $\delta$.
Moreover, $a(P)\in M$ and since $\bar\delta = \sup (M\cap \delta)$ we have that $\bar\delta \in a(P)$.
This implies that $\M_{p\restriction M}^\gamma \rest \bar\delta = \M_p^{\bar\delta}\rest M$. 
Therefore, by \cref{gapmeet} it is a $\bar\delta$-chain closed under active meets. 
Now, suppose $N,P\in \M_{p\restriction M}^\delta$ and $N\in_{\bar\delta} P$.
Since $\bar\delta = \sup (M\cap \delta)$ we have that $N\in_\xi P$, for unboundedly many $\xi \in E\cap \delta$. 
We conclude that $N\in_\delta P$. Indeed, if $P$ is countable this follows from \cref{incontcount} applied in $A$, and
if $P$ is a Magidor model this follows from \cref{incont}, again applied in $A$.
Moreover, assuming $N$ is a Magidor model and $P$ is countable, and $N\rest \gamma \land P\rest \gamma$ is defined and active at $\gamma$
then, by \cref{MPI}, $N\land P$ is defined and active at unboundedly many $\xi \in E\cap \delta$, and hence it is also active at $\delta$ and $\bar\delta$. 
It follows that $\M_{p\restriction M}^\delta$ is a $\delta$-chain closed under meets, and hence $\M_{p\restriction M}^\gamma$
is  a $\gamma$-chain closed under meets as well. 

Let us check that every $P^*\in \dom(d_{p\restriction M})$ is $\M_{p\restriction M}$-free. 
If $P^*  \in \dom(d_p)$ this is immediate. 
Otherwise, $P^*$ is of the form $P\rest M$, for some $P\in \dom(d_p)$
such that $\eta(P)\notin M$. Let $\eta=\eta(P)$ and $\eta^*=\eta(P^*)$. Note that $M$ is strongly active at $\eta$
and $\eta^*$ is the least ordinal of $M$ above $\eta$. Suppose $N\in \M_{p\restriction M}$ is such that $P^*\in_{\eta^*} N$. 
Then $N\rest \eta\in \mathcal L(\M_p)$ and $P\in_\eta N$. Since $P$ is $\M_p$-free, $N$ must be strongly active at $\eta$. 
Since $\eta= \sup (M\cap \eta^*)$ and $N\in M$ we must have that $N$ is strongly active at $\eta^*$ as well. 
This also establishes $(*)$ from \cref{PF}. Indeed, if $P^*\in_{\eta^*} N$ then $P\in_\eta N\rest \eta$, and
hence $d_p(P)\subseteq N$, since $N\rest \eta \in \mathcal L(\M_p)$, and $p$ is a condition. 
This completes the proof that $p\rest M= (\mathbb P^\kappa_{\alpha^*})^A$.
\end{proof}

\begin{remark} Note that if $p,q\in \mathbb P^\kappa_\lambda$ are such that $q\leq p$ and $M\in \M_p$ then 
$q\rest M \leq p\rest M$. 
\end{remark}

We are planning to show that if $p$ is a condition and $M\in \M_p$ is a countable $\beta$-model
then, for  any $q\leq p\rest M$ with $q\in M$, $p$ and $q\rest \beta$  are compatible, and in fact
the meet $p\land q\rest \beta$ exists. 
Before that we show the following special case of this statement.

\begin{lemma}\label{partial-countable-merging} Suppose $p\in \mathbb P^\kappa_\lambda$ and $\delta \in E$. 
Suppose $M\in \M_p^\delta$ is a countable model,  $\M\in_\delta M, \M\in \mathbb M^\kappa_\delta$, 
and  $(\M_p\rest M)^\gamma \subseteq \M^\gamma$, for all $\gamma \in E\cap (\delta +1)$. 
Suppose further that $P\notin_{\eta(P)}M$, for all $P\in \dom(d_p)$. 
Let $\M_q$ be the closure of $\M_p \cup \M$ under meets and let $d_q=d_p$. Finally, let $q=(\M_q,d_q)$.
Then $q\in \mathbb P^\kappa_\lambda$. 
\end{lemma}

\begin{proof} Let us first check that $\M_q^\gamma$ is a $\gamma$-chain closed under active meets, for all $\gamma \in E$.
Fix $\gamma \in E$.  If $M$ is not active at $\gamma$ then $\M_q^\gamma =\M_p^\gamma$, so this follows from the fact that $p$ is a condition. 
 If $M$ is active at $\gamma$ then  this follows from \cref{level-merging}. 

Thus, it remains to check that every $P\in \dom (d_p)$ is $\M_q$-free
and $d_p(P)\in Q$, for all $Q\in \M_q$ such that $P\in_{\eta(P)}Q$. 
Now, fix one such $P\in \dom(d_p)$ and let $\eta = \eta(P)$.
If $M$ is not active at $\eta$, then no model of $\M$ is active at $\eta$, and hence $Q\in \M_p$, for all $Q\in \M_q$ such that $P\in_\eta Q$. 
The conclusion then follows from the fact that $p$ is a condition and $d_p$ is its decoration. 
Suppose now that $M$ is active at $\eta$, but $P$ is either equal to $M\rest \eta$ or is above $M\rest \eta$ on the $\eta$-chain $\M_p^\eta$.
Then, again any $Q\in \M_q$ such that $P\in_\eta Q$ is in $\M_p$, and the conclusion follows as above. 
Suppose now that $M$ is active at $\eta$ and $P\in_\eta^*M$. Note that $\M_q^\eta$ is obtained by
closing $\M_p^\eta \cup \M^\eta$ under meets that are active at $\eta$. 
Suppose $P$ is below $M\rest \eta$ on $\M_p^\eta$. By the assumption, $P\notin_\eta M$, hence by \cref{gapmeet}, 
there must be a Magidor model $N\in (\M_p\rest M)^\eta$ such that $P$ is in the interval $[N\land M\rest \eta ,N)_p^\eta$.
By \cref{intermediate}, we have that $P\in_\eta N$ and thus $d_q(P)\in N$. 
Note that $P$ also belongs to the interval $[N\land M\rest \eta,N)_q^\eta$.
Suppose $Q\in \M_q$ and $P\in_\eta Q$. By replacing $Q$ with $Q\rest \eta$, we may assume that $Q\in \M^\eta$.
Note that $Q$ cannot be a countable model since then we would have $P\in_{\eta}M$. 
If $Q$ is a Magidor model in $\M^\eta$ then $Q$ cannot be below $N$ since then it would be below $N\land M\rest \eta$ on $\M_q^\eta$.
Therefore, $Q$ must be either equal to $N$ or above $N$ on the $\eta$-chain  $\M^\eta$. 
Then we would have $N\cap V_\kappa \subseteq Q \cap V_\kappa$, and hence $d_p(P)\in Q$.
If $Q\in \M_p$ then $Q$ is strongly active at $\eta$ and $d_p(P)\in Q$, since $p$ is a condition. 
It remains to consider the case when $Q$ is of the form $R\land S$, for some Magidor model $R\in \M^\eta$
and countable $S\in \M_p^\eta \setminus \M^ \eta$.
Now, we must have $R=N$ or $N\in_\eta R$ since otherwise $R$, and hence also  $R\land S$, would be below $N\land M\rest \eta$.
Since $d_p(P)\in N$, we must have $d_p(P)\in R$. 
Moreover, since $S \in \M_p^\eta$, and $d_p$ is the decoration of $p$, 
$S$ must be strongly active at $\eta$ and $d_p(P)\in S$. 
By \cref{meet-strongly-active}, $R\land S$ is strongly active at $\eta$. 
By \cref{meet-trace}, we have $R\land S \cap V_\eta = R\cap S\cap V_\eta$, and hence $d_p(P)\in R\land S=Q$. 
\end{proof}

\begin{lemma}\label{SP} 
Suppose $p\in\mathbb P^\kappa_\lambda$ and $M\in {\mathcal L}(\M_p)$ is a countable $\beta$-model, for some $\beta \in E$.
Let $\alpha^*\in M$ be such that $p\rest M\in (\mathbb P^\kappa_{\alpha^*})^M$.
Then for any $q\in (\mathbb P^\kappa_{\alpha^*})^M$ with $q\leq p\rest M$, $p$ and $q\rest \beta$ are compatible,
and the meet $p\land q\rest \beta$ exists. 
\end{lemma}

\begin{proof} Let $\M_r$ be the closure of $\M_p \cup \M_{q\restriction \beta}$ under meets. By \cref{level-merging} we already know
that $\M_r^\delta$ is a $\delta$-chain, for all $\delta \in E$. Hence $\M_r\in \mathbb M^\kappa_\lambda$. 
It remains to define the decoration $d_r$, and check that it satisfies $(*)$ from \cref{PF}.
Let 
\[
D_p= \{ P\in \dom(d_p) : P\notin_{\eta(P)} M\}.
\]
 Now, suppose $P\in \dom(d_q)$. 
 Let $\delta (P)$ be the largest ordinal $\gamma\in E\cap (\eta(P)+1)$ such that $M$ is strongly active at $\gamma$. 
 Let 
 \[
 D_q= \{ P\rest \delta(P) : P\in \dom (d_q)\}.
 \]
Note that, for every $P\in \dom(d_q)$, we have $(P \rest \delta(P))\rest M = P$, and $P$ is active at $\delta(P)$. 
Observe that $D_p$ and $D_q$ are disjoint. 
Let $\dom(d_r)= D_p \cup D_q$ and define $d_r$ by: 
\[
d_r(P) = \begin{cases} 
             d_p(P) &\mbox{if } P\in D_p \\
             d_q(P\rest M) & \mbox{if } P\in D_q \text{ and } \eta(P)\leq \beta 
             \end{cases}
\]
We have to check that every $P\in \dom(d_r)$ is $\M_r$-free and condition $(*)$  holds. 
By \cref{partial-countable-merging} we have that $(\M_r,d_p\rest D_p)$ is already a condition,
so we may assume $P\in D_q$. 
Fix one such $P\in D_q$, and  let $\eta = \eta(P)$.  
Note that it suffices to show that the least model, say $R$,  on the $\eta$-chain $\M_r^\eta$ above $P$
is strongly active at $\eta$, and $d_r(P)\in R$. 
By  \cref{gapmeet} either $R\in_\eta M$ or $R=N\land M\rest \eta$, for some Magidor model $N\in (\M_p\rest M)^\eta$. 
Now, if $R$ is of the form $N\land M$, then, since $N$ and  $M$ are strongly active at $\eta$, by \cref{meet-strongly-active}, so is $R$. 
Moreover, $N\rest M\in \mathcal L(\M_q)$ and $d_q(P\rest M)\in M\cap N$. It follows that $d_r(P)\in R$.
Suppose now that  $R\in_\eta M$. Let $\rho= \min (E\cap M\setminus \eta)$. 
Then $P\rest M$ and $R\rest M$ are $\rho$-models, $R\rest M\in \mathcal L(\M_q)$, and $P\rest M \in_\rho R\rest M$. 
Therefore, $R\rest M$ is strongly active at $\rho$, and $d_q(P\rest M)\in R\rest M$. 
Since $(R\rest M) \cap V_\kappa = R\cap V_\kappa$, we get that $d_q(P\rest M)\in R$, and hence $d_r(P)\in R$. 
Moreover, since $R\rest M$ is strongly active at $\rho$, it follows that $R$ is strongly active at $\eta$. 
This shows that all the models in $\dom(d_r)$ are $\M_r$-free and condition $(*)$ holds for $r$. 
The fact that $r\leq p,q\rest \beta$ and is in fact the weakest such condition follows from the definition.  
\end{proof}

\begin{remark}\label{remark-standard} Suppose $p\in \mathbb P^\kappa_\lambda$ and $M\in \M_p$ is a countable $\beta$-model,
for some $\beta \in E$. If either $M$ is standard or $\beta\in M$ we have that $p \rest M\in \mathbb P^\kappa_\lambda$. 
 In particular, \cref{SP} shows that if $p\in \mathbb P^\kappa_\lambda$ then $p$ and $p\rest M$ 
are compatible. Now, we have already observed that, if  $q\in \mathbb P^\kappa_\lambda$ and $q\leq p$,
then $q\rest M \leq p\rest M$. Therefore, even though  it may not be the case that $p\leq p\rest M$,
 every $p$ forces $p\rest M$ to belong to the generic filter, and hence $p$ is stronger than $p\rest M$. 
\end{remark}

Now, by \cref{topM} and \cref{SP} we immediately get the following.

\begin{theorem}\label{countable-sp}
$\mathbb P^\kappa_\lambda$ is $\mathscr C_{\rm st}$-strongly proper.
\end{theorem} 
\begin{proof}
Suppose $M\in\mathscr C$ and $p\in M\cap\mathbb P^\kappa_\lambda$. Let $p^M$ be the condition defined
in \cref{topM}. If $q\leq p^M$ then $M\in {\mathcal L}(\M_q)$ and $q\rest M\in M$, and by \cref{remark-standard}, $q\restriction M \in \mathbb P^\kappa_\lambda$. 
Then, by \cref{SP} any extension $r$ of $q\rest M$ with $r\in M$ 
is compatible with $q$, and moreover $q\land r$ exists. 
Thus, $p^M$ is a $(M,\mathbb P^\kappa_\lambda)$-strongly generic condition extending $p$.
\end{proof}

\begin{remark}\label{alpha-strongly-proper}
A  similar proof shows that the forcing $\mathbb P^\kappa_\alpha$ is strongly proper for the collection of
all  $M\in\mathscr C$ such that $\alpha\in M$. 
\end{remark}

\begin{notation}
Let $F$ be a filter in $\mathbb P^\kappa_\lambda$. Then we let $\mathcal M_{F}$ denotes $\bigcup\{\mathcal M_p:~p\in F\}$.
\end{notation}

Let $G$ be a $\mathbb P^\kappa_\lambda$-generic filter over $V$. 
We let $G_\alpha=G\cap\mathbb P^\kappa_\alpha$, for all $\alpha\in E$. 
The following is straightforward. 

\begin{proposition}\label{chainfilter}
 Let $\delta \in E$ with ${\rm cof}(\delta)<\kappa$. Then $\mathcal M_G^\delta$ is a $\delta$-chain.
\end{proposition}
\begin{proof}[\nopunct] 
 \end{proof}

\begin{proposition}\label{continuity-Magidor} Let $\delta\in E$  with ${\rm cof}(\delta)<\kappa$. 
Suppose $M\in \M_G^\delta$ is a Magidor model and is not the least model in $\M_G^\delta$.
Then  $M\cap V_\delta = \bigcup \{ Q\cap V_\delta : Q\in_\delta M \mbox{ and } Q\in \M_G^\delta\}$. 
\end{proposition}

\begin{proof} It suffices to show that if $p\in \mathbb P^\kappa_\lambda$ and $M\in \M_p^\delta$ is a Magidor model
that is not the least model of $\M_p^\delta$, and $x\in M\cap V_\delta$, then there is $q\leq p$
and $Q\in \M_q$ which is active at $\delta$ such that  $Q\in_\delta M$ and $x\in Q$. 
Let $N\in \M_p$ be active at $\delta$ such that $N\in_\delta M$. 
Fix some $Q^*\in \mathscr C$ such that $N,M,x\in Q^*$. By Lemma \cref{topM} there
is a condition $q\leq p$ such that $Q^*\in \M_q$. Since $N\in_\delta M$ and $N\in_\delta Q^*$,
if we let $Q= M\wedge Q^*$, by \cref{MPI} $N\in_\delta Q$, and hence $Q$ is active at $\delta$. 
Moreover, $Q\in_\delta M$ and by \cref{meet-trace} $Q\cap V_\delta = N\cap Q^* \cap V_\delta$ and hence
$x\in Q$. It follows that the condition $q$ and the model $Q$ are as required. 
\end{proof}

\begin{theorem} Assume $\kappa$ is supercompact. 
Then $\mathbb P^{\kappa}_\lambda$ preserves $\omega_1$ and $\kappa$, and collapses all cardinals 
between $\omega_1$ and $\kappa$ to $\omega_1$. 
\end{theorem}
\begin{proof}
By  \cref{cstationaryinV}, $\mathscr C_{\rm st}$ is stationary  in ${\mathcal P}_{\omega_1}(V_\lambda)$, and 
by \cref{SP}, $\mathbb P^\kappa_\lambda$ is $\mathscr C_{\rm st}$-strongly proper. Hence $\omega_1$ is preserved. 
By \cref{ustationaryinV}, $\mathscr U$ is stationary in ${\mathcal P}_{\kappa}(V_\lambda)$, and by \cref{Magidor-strongly-proper},
$\mathbb P^\kappa_\lambda$ is $\mathscr U$-strongly proper. Hence $\kappa$ is preserved. 
Now,  fix  a  cardinal $\mu<\kappa$. Let $G$ be a $\mathbb P^\kappa_\lambda$-generic filter over $V$. 
Fix  $\alpha\in E$ of cofinality less than $\kappa$. 
A standard density argument shows that there exists a Magidor  model $N\in \M_G^\alpha$  such that $\mu\in N$. 
By \cref{chainfilter}  $\mathcal M_G^\alpha$ is an $\in_\alpha$-chain. 
Let $N^*$ be the least Magidor model above $N$ in $\mathcal M_G^\alpha$, and let  $I=(N,N^*)^\alpha_{\M_G}$.
Note that every model in $I$ is countable and $\in_\alpha$ is transitive on $I$. 
Hence if $P,Q\in I$ and $P\in_\alpha Q$ then $P\cap V_\alpha \subseteq Q\cap V_\alpha$. 
Another standard density argument shows that, for every $x\in N\cap V_\alpha$, there is $P\in I$ such that 
$x\in P$. Thus, $\{ P\cap V_\alpha: P\in I\}$ is an increasing chain of countable sets whose union covers $N\cap V_\alpha$. 
It follows that $N\cap V_\alpha$ is of cardinality at most $\omega_1$. Since $\mu$ belongs to the transitive part of $N$,
we also get that $| \mu | \leq \omega_1$. 
\end{proof}

\begin{theorem}
$\mathbb P^{\kappa}_\lambda$ collapses cardinals of the interval between $\kappa$ and $\lambda$ to $\kappa$.
\end{theorem}
\begin{proof}
Let $\alpha \in E$ be of cofinality less than $\kappa$, and let $G$ be a $V$-generic filter over $\mathbb P^\kappa_\lambda$.. 
Let $U_\alpha$ be the set of Magidor models in $\M_G^\alpha$. By \cref{intermediate}, we have that 
$\in_\alpha$ is transitive on $U_\alpha$. Note that if $P,Q \in U_\alpha$ then $P\cap V_\alpha  \subseteq Q\cap V_\alpha$. 
Now, a standard density argument using the stationarity of $\mathscr U$ shows that, for every $x\in V_\alpha$, there is $P\in U_\alpha$ such that $x\in P$. 
It follows that $\{ P\cap V_\alpha : P\in U_\alpha\}$ is an increasing family of sets of size $<\kappa$ whose union is $V_\alpha$.
Therefore, $V_\alpha$ has cardinality $\leq \kappa$ in $V[G]$.   
\end{proof}

\begin{theorem}
Suppose $\lambda$ is an inaccessible cardinal. Then $\mathbb P^{\kappa}_\lambda$ is $\lambda$-c.c.
\end{theorem}
\begin{proof}
For each $p\in \mathbb P^\kappa_\lambda$, let $a(p)=\bigcup \{ a(M): M\in \M_p\}$. 
Note that $a(p)$ is a closed subset of $E$ of size $<\kappa$, for all $p$. 
Suppose $A$ is a subset of $P^{\kappa}_\lambda$ of cardinality $\lambda$. 
Since $\lambda$ is inaccessible, by a standard $\Delta$-system argument, we can find a subset $B$ of $A$ of size $\lambda$ 
and a subset $a$ of $E$ such that $a(p)\cap a(q)= a$, for all distinct $p,q\in B$. 
Note that $a$ is closed, and if we let $\gamma = \max(a)$ then $\gamma \in E$.
Since $B$ has size $\lambda$, by a simple counting argument, we may assume there is $\M\in \mathbb M^\kappa_\gamma$
such that $\M_p\rest \gamma = \M$, for all $p\in B$. 
Now, pick distinct $p,q\in B$, and define $\M_r= \M_p \cup \M_q$ and $d_r= d_p \cup d_q$. 
Let $r=(\M_r,d_r)$. It is straightforward to check that $r\in \mathbb P^\kappa_\lambda$ and $r\leq p,q$. 
\end{proof}

\begin{definition}\label{def-clubs} Suppose $G$ is $V$-generic over $\mathbb P^\kappa_\lambda$ and $\alpha \in E$ is of cofinality less than $\kappa$. 
Let $C_\alpha(G)= \{ \kappa_M : M\in \M_G^\alpha\}$. 
\end{definition}

\begin{lemma}\label{addingclub}
Let $G$ be a $V$-generic filter over $\mathbb P^\kappa_\lambda$. Then $C_\alpha(G)$ is a club in $\kappa$, for all $\alpha \in E$
of cofinality $<\kappa$.  Moreover, if $\alpha < \beta$ then $C_\beta(G)\setminus C_\alpha(G)$ is bounded in $\kappa$. 
\end{lemma}
\begin{proof} 
Let us check the second statement first. Fix $\alpha, \beta \in E$ such that $\cof(\alpha), \cof(\beta) <\kappa$, and $\alpha <\beta$.
By a standard density argument using the stationarity of $\mathscr U$ there is $p\in G$ and a Magidor model $M\in \M_G$ which is active at both $\alpha$ and $\beta$. 
Therefore, any model $N$ above $M\rest \beta$ on the $\beta$-chain $\M_G^\beta$ is also active at $\alpha$. 
It follows that $C_\beta(G)  \setminus C_\alpha(G)  \subseteq \kappa_M$. 

We work in $V$ and prove the first statement by induction on $\alpha$. 
Let $\dot{\M}^\alpha$ and $\dot{C}_\alpha$ be canonical $\mathbb P^\kappa_\lambda$-names for $\M_G^\alpha$ and $C_\alpha(G)$, for  $\alpha \in E$.
Now, fix  $\alpha \in E$ of cofinality less than $\kappa$ and suppose
the statement has been proved for all $\bar\alpha \in E\cap \alpha$ of cofinality $<\kappa$. 
Suppose $\gamma <\kappa$ and $p\in \mathbb P^\kappa_\lambda$ forces that $\gamma$ is a limit point but not a member of $\dot{C}_\alpha$ 
We may assume that there is a model $M\in \M_p^\alpha$ such that $p$ forces that $M$ is the least model on the $\alpha$-chain $\dot{\M}^\alpha$
such that $\gamma \leq   \kappa_M$. Then we must have $\gamma < \kappa_M$. Let $P$ be the previous model on $\M_p^\alpha$ before $M$. 
We may assume that such a model exists since $p$ forces that $\gamma$ is a limit point of $\dot{C}_\alpha$.  
Note that $\kappa_P <\gamma$ since $p$ forces that $\gamma\notin \dot{C}_\alpha$.

\noindent {\bf Case 1.} Suppose $M$ is strongly active at $\alpha$.
Since  $P$ is $\M_p$-free and we may assume that $P\in \dom(d_p)$, by defining $d_p(P)=\emptyset$ if necessary. 
Since $\gamma < \kappa_M$, we can find $\delta \in M$ such that $\gamma \leq \delta < \kappa_M$. 
Define a condition $q$ as follows. Let $\M_q=\M_p$, and let $\dom (d_q)=\dom(d_p)$. 
Let $d_q(P)= d_p(P)\cup \{ \delta\}$, and $d_q(Q)=d_p(Q)$, for any other $Q\in \dom (d_p)$. 
Let $q= (\M_q,d_q)$. Then $q$ is a condition and forces that the next model of $\dot{\M}^\alpha$ above $P$ contains $\delta$.
Hence, it  forces that there is no element of $\dot{C}_{\alpha}$ between $\kappa_P$ and $\gamma$, 
and so it forces that $\gamma$ is not a limit point of $\dot{C}_\alpha$, a contradiction. 

\noindent {\bf Case 2.} Suppose now that $M$ is not strongly active at $\alpha$. Then $M$ is countable.
Let $A={\rm Hull}(M,V_\alpha)$, let $\alpha^*$
be the least ordinal of $M$ above $\alpha$, and let $\bar\alpha= \sup (M\cap \alpha)$. 
Note that $\alpha^*\in E_A$, $\bar\alpha$ is a limit point of $E$ of cofinality $\omega$, and that $P$ is also active at $\bar\alpha$.
Now, by the proof of the second part of the lemma,  $p$  forces that $\dot{C}_\alpha \setminus \dot{C}_{\bar\alpha} \subseteq \kappa_P$,
and so it also forces that $\gamma$ is a limit point of $\dot{C}_{\bar\alpha}$.
By the inductive assumption $\dot{C}_{\bar\alpha}$ is forced to be a club, so there is $q\leq p$ and some $N\in \M_q^{\bar\alpha}$ such that $\kappa_N= \gamma$. 
Now,  for each $Q\in (\M_q\rest M)^{ \bar\alpha}$, we can find a unique model $Q^*\in M$ with $Q^*\in \mathscr V_{\alpha^*}^A$
such that $Q^*\rest \bar\alpha = Q$. Let $\M^*= \{ Q^*: Q \in (\M_p\rest M)^{\bar\alpha}\}$. 
Working in $A$, $\M^*$ is an $\alpha^*$-chain closed under meets that are active at $\alpha^*$.
Let $\M= \{ Q^*\rest \alpha : Q^*\in \M^*\}$. Then $\M\in_\alpha M$, is an $\alpha$-chain closed under meets
that are active at $\alpha$, and $(\M_q\rest M)^\alpha \subseteq \M$. 
We now define a condition $r$. Let $\M_r$ be the closure of $\M_p$ and $\M$ under meets. 
By applying \cref{level-merging}, for all levels $\delta \in E\cap (\bar\alpha,\alpha]$, we have that $\M_r\in \mathbb M^\kappa_\lambda$. 
Let $d_r=d_q$ and $r=(\M_r,d_r)$. Observe that $\M_r^\eta = \M_q^\eta$, 
for all  $\eta \in E \setminus (\bar\alpha,\alpha]$.
Also, if $R\in \dom (d_q)$ and $\eta(R)\in (\bar\alpha, \alpha]$ then $R\notin_{\eta(R)}M$, since $M$ is not strongly active at $\eta (R)$.
By \cref{partial-countable-merging}, we conclude that $r$ is a condition. Also, we have that $r\leq q$. 
Recall that $N\in \M_q^{\bar\alpha}$ and $\kappa_N=\gamma$.
Let $Q$ be the model on the $\bar\alpha$-chain $\M_r^{\bar\alpha}$ immediately before $M\rest \bar\alpha$.  
Then $Q^* \in \M^*$, and hence $Q^*\rest \alpha \in \M_r$. Let $R=Q^*\rest \alpha$. 
In other words, we lifted the model $Q$ to level $\alpha$ and called this model $R$. 
Note that $\kappa_R= \kappa_Q$. 
Then $r$ forces that $R\in \dot{\M}_\alpha$ and $\gamma \leq \kappa_R < \kappa_M$, which contradicts the fact
that $p$ forces that $\gamma \notin \dot{C}_\alpha$ and $M$ is the least model on $\dot{\M}_\alpha$ with $\gamma \leq  \kappa_M$. 
This completes the proof of the lemma. 
\end{proof}

\section{Guessing Models in $V[G]$}

We assume $\kappa$ is supercompact and $\lambda$ is inaccessible and analyze $\omega_1$-guessing models 
in the the generic extension by $\mathbb P^\kappa_\lambda$. 
Suppose $\alpha \in E$. We have already established in \cref{alpha-complete-suborder}  that $\mathbb P^\kappa_\alpha$
is a complete suborder of $\mathbb P^\kappa_\lambda$. 
Let us fix  a $V$-generic filter $G_\alpha$ over $\mathbb P^\kappa_\alpha$, and let $\mathbb Q_\alpha$  denote the quotient forcing. 
Recall that  $\mathbb Q_\alpha$ consists of all $p\in \mathbb P^\kappa_\lambda$ such that  $p\rest \alpha \in G_\alpha$,
with the induced ordering. Forcing with this poset over $V[G_\alpha]$ produces a $V$-generic filter $G$ for $\mathbb P^\kappa_\lambda$
such that $G\cap \mathbb P^\kappa_\alpha= G_\alpha$. 
We first show that the pair $(V[G_\alpha],V[G])$ has the $\omega_1$-approximation property. 
We will need the following definition. 

\begin{definition}Let $\mathscr C^\alpha_{\rm st}$ denote the set of all $M\in \mathscr C_{\rm st}$ such  that $\eta(M)> \alpha$, $\alpha \in M$,
and  $M\rest \alpha \in \M_{G_\alpha}$. 
\end{definition}

\begin{lemma}\label{C-stationary} $\mathscr C^\alpha_{\rm st}$ is stationary subset of ${\mathcal P}_{\omega_1}(V_\lambda)$
in the model $V[G_\alpha]$. 
\end{lemma}
\begin{proof}
We work in $V$. Let $\dot{\mathscr C}^\alpha_{\rm st}$ and $\dot{\M}_{\alpha}$ be the canonical $\mathbb P^\kappa_\alpha$-names
for $\mathscr C^\alpha_{\rm st}$ and $\M_{G_\alpha}$.
Suppose $p\in \mathbb P^\kappa_\alpha$ forces that  $\dot F: [V_\lambda]^{<\omega}\rightarrow V_\lambda$ is a function.
Let $\theta$ be a sufficiently large regular cardinal. By \cref{cstationaryinV}, $\mathscr C_{\rm st}$ is club in $\mathcal P_{\omega_1}(V_\lambda)$,
hence we can find a countable $M^*\prec H(\theta)$  containing all the relevant objects such that
letting $M=M^*\cap V_\lambda$ we have that $M \in \mathscr C_{\rm st}$. 
Let $M'= M\rest \alpha$.  Note that $p\in M'$, so we can form the condition $p^{M'}$.
Then  $p^{M'}$  is $(M',\mathbb P^\kappa_\alpha)$-strongly generic and $p^{M'}\leq p$. 
 Let $\sigma$ be the $\alpha$-isomorphism between $M$ and $M'$. 
Note that $\sigma(q)=q$, for all $q\in M\cap \mathbb P^\kappa_\alpha$. 
Hence, $M \cap \mathbb P^\kappa_\alpha=  M' \cap \mathbb P^\kappa_\alpha$.
Therefore, $p^{M'}$ is also $(M,\mathbb P^\kappa_\alpha)$-strongly generic, and thus it is $(M^*,\mathbb P^\kappa_\alpha)$-generic.
Since $\dot F\in M^*$, it follows that $p^{M'}$ forces that $M$ is closed under $\dot F$.
It also forces that $M'$ belongs to $\dot{\M}_{\alpha}$, hence it forces that $M$ belongs to $\dot{\mathscr C}^\alpha_{\rm st}$. 
\end{proof}

\begin{lemma}\label{quotient-alpha-sp} Suppose $\alpha \in E$ and let $G_\alpha$ be $V$-generic over $\mathbb P^\kappa_\alpha$. 
Then $\mathbb Q_\alpha$ is $\mathscr C^\alpha_{\rm st}$-strongly proper.
\end{lemma}

\begin{proof} 
Work in $V[G_\alpha]$. Let $p\in \mathbb Q_\alpha$, and $M\in \mathscr C^\alpha_{\rm st}$ be such that $p\in M$. 
Let $p^M$ be the condition defined in \cref{topM}. Since $\alpha \in M$ we have $p\rest \alpha \in M$, and also $p\rest \alpha \in M\rest \alpha$.
Note that $p^M\rest \alpha = (p\rest \alpha)^{M\restriction \alpha}$. 
Since $p\rest \alpha \in G_\alpha$ and $M\rest \alpha \in \M_{G_\alpha}$, we have that $p^M\rest \alpha \in G_\alpha$,
thus $p^M\in \mathbb Q_\alpha$. 
Let us show that $p^M$ is $(M,\mathbb Q_\alpha)$-strongly generic. 
Suppose $q\leq p^M$ and $q\rest \alpha \in G_\alpha$. 
Since $\alpha \in M\rest \alpha$ we have  $(q\rest M)\rest \alpha= (q\rest \alpha)\rest (M\rest \alpha)$, 
and hence $q\rest M \in M \cap \mathbb Q_\alpha$.
Let $r\leq q\rest M$ be such that $r\in M \cap \mathbb Q_\alpha$.
By \cref{SP}, $r$ and $q$ are compatible in $\mathbb P^\kappa_\lambda$ and the meet $r\land q$ exists. 
Now, observe that the meet of $r\rest \alpha$ and $q\rest \alpha$ exists, and 
$r\rest \alpha \land q\rest \alpha =(r\land q)\rest \alpha$. Since $r\rest \alpha, q\rest\alpha \in G_\alpha$,
we conclude that $r\rest \alpha \land q\rest \alpha  \in \mathbb Q_\alpha$. 
It follows that $q$ and $r$ are compatible  in $\mathbb Q_\alpha$.
\end{proof}

Now, by \cref{C-stationary}, \cref{quotient-alpha-sp}, and \cref{guessingbystronglyproper}, we get the following. 

\begin{corollary}\label{approx-alpha} The pair $(V[G_\alpha],V[G])$ has the $\omega_1$-approximation property. 
\end{corollary}

\begin{proof}[\nopunct] 
 \end{proof}

Suppose now $N\in \mathscr U$.
Let ${\mathbf1}^N=(\{ N\}, \emptyset)$. By \cref{Magidor-compatible}, ${\mathbf 1}^N$ is
$(N,\mathbb P^\kappa_\lambda)$-strongly generic. Moreover, for every $q\leq {\mathbf 1}^N$ and $r\leq q\rest N$,
$q$ and $r$ are compatible, and the meet $q\land r$ exists. Let $\mathbb P_N= \mathbb P^\kappa_\lambda \cap N$
and let 
\[
\mathbb P^\kappa_\lambda \rest N = \{ q\in \mathbb P^\kappa_\lambda : N\in \M_q\}.
\] 
Then the map $p\mapsto p^N$ is a complete embedding from $\mathbb P_N$ to $\mathbb P^\kappa_\lambda \rest N$. 
Now, fix a $V$-generic filter $G_N$ over $\mathbb P_N$. 

\begin{definition} Let $G_N$ be a $V$-generic filter over $\mathbb P_N$. 
Let $\mathscr C^N_{\rm st}$ denote the set of all $M\in \mathscr C_{\rm st}$ such  that $N\in M$ and $N\land M\in \M_{G_N}$.
\end{definition}

\begin{lemma}\label{stationaryquotient}
The collection $\mathscr C^N_{\rm st}$ is stationary in $\mathcal P_{\omega_1}(V_\lambda)$ in the model $V[G_N]$.
\end{lemma}

\begin{proof}
This is very similar to the proof of \cref{C-stationary}. 
We work in $V$. Let $\dot{\mathscr C}^N_{\rm st}$ be the canonical $\mathbb P_N$-name for $\mathscr C^N_{\rm st}$.
Suppose $p\in \mathbb P_N$ forces that  $\dot F: [V_\lambda]^{<\omega}\rightarrow V_\lambda$ is a function.
Let $\theta$ be a sufficiently large regular cardinal. By \cref{cstationaryinV}, $\mathscr C_{\rm st}$ is stationary in $\mathcal P_{\omega_1}(V_\lambda)$,
hence we can find a countable $M^*\prec H(\theta)$  containing all the relevant objects.
Let $M=M^*\cap V_\lambda$, and note that $M \in \mathscr C_{\rm st}$. Since $N\in_{\eta(N)}M$, the meet $N\land M$ is defined.
Let $\eta = \eta (N\land M)$ and let $\sigma$ be the $\eta$-isomorphism between $N\cap M$ and $N\land M$.
Note that $\sigma (q)=q$, for all $q\in \mathbb P_N$. 
Now, $p^{N\land M}$ is $(N\land M, \mathbb P_N)$-strongly generic, hence also $(N\cap M, \mathbb P_N)$-strongly generic, 
and therefore it is $(M^*,\mathbb P_N)$-generic. It follows that $p^{N\land M}$ forces that $M\in {\dot{\mathscr C}}^N_{\rm st}$ and
is closed under $\dot F$. 
\end{proof}

Let $\mathbb Q_N$ denotes the quotient forcing $(\mathbb P^\kappa_\lambda\rest N)/G_N$.

\begin{lemma}\label{stronglyproperquotient}
$\mathbb Q_N$ is $\mathscr C^N_{\rm st}$-strongly proper.
\end{lemma}

\begin{proof}
Work in $V[G_N]$. Let $p\in \mathbb Q_N$ and  $M\in \mathscr C^N_{\rm st}$ be such that $p\in M$. 
Let $p^M$ be the condition defined in \cref{topM}. Since $p,N\in M$, we have $p\rest N \in M$. 
Thus, $p\rest N \in N\cap M$. 
Observe that $p^M\rest N= (p\rest N)^{N\land M}$.  
Since $p\rest N \in G_N$ and $N\land M \in \M_{G_N}$, we have that $p^M\rest N \in G_N$,
thus $p^M\in \mathbb Q_N$. 
Let us show that $p^M$ is $(M,\mathbb Q_N)$-strongly generic. 
Suppose $q\leq p^M$ and $q\in \mathbb Q_N$.  
Observe that $(q\rest M)\rest N = (q\rest N)\rest (N\land M)$, and hence $q\rest M \in \mathbb Q_N$.
Let $r\leq q\rest M$ be such that $r\in M\cap \mathbb Q_N$.  
By \cref{SP}, $r$ and $q$ are compatible in $\mathbb P^\kappa_\lambda$ and the meet $r\land q$ exists. 
Note that  $r\rest N\in N\cap M \subseteq N\land M$,  and $r\rest N$ extends $(q\rest N)\rest (N\land M)$. 
Hence, again by \cref{SP}, the meet of $r\rest N$ and $q\rest N$ exists,
and  $r\rest N \land q\rest N = (r\land q)\rest N$. 
Since $r\rest N, q\rest N \in G_N$,  we have that $r\rest N \land q\rest N \in G_N$. 
It follows that $r$ and $q$ are compatible in $\mathbb Q_N$. 
\end{proof}

Suppose $G$ is a $V$-generic filter over  $\mathbb P^\kappa_\lambda$. 
As before, for $\alpha \in E$,  let $G_\alpha= G\cap \mathbb P^\kappa_\alpha$.

\begin{lemma}\label{Magidor-guessing}
Let $\alpha \in E$. Suppose $N\in \M_G$ is a Magidor model with $\alpha \in N$.
 Then $N[G_\alpha]$  is an $\omega_1$-guessing model in $V[G]$. 
\end{lemma}

\begin{proof} 
 Let $\overline{N}$ be the transitive collapse of $N$, and let $\pi$ be the collapsing map. 
 For convenience, let us write $\bar\kappa$ for $\kappa_N$. 
Then $\overline{N}=V_{\bar\gamma}$, for some  $\bar\gamma$ with $\cof(\bar\gamma)\geq \bar\kappa$ 
and $\pi(\kappa)= \bar\kappa$. Let $\bar\alpha= \pi(\alpha)$. 
Note that $\pi (\mathbb P_N)= \mathbb P^{\bar\kappa}_{\bar\alpha}$. 
Let $G_N= G_\alpha \cap N$ and  $G^{\bar\kappa}_{\bar\alpha}= \pi [G_N]$. 
Then the transitive collapse $\overline{N[G_N]}$ of $N[G_N]$ is equal to $V_{\bar\gamma}[G^{\bar\kappa}_{\bar\alpha}]= V_{\bar\gamma}^{V[G_N]}$.
Hence $N[G_N]$ is an $\omega_1$-guessing model in $V[G_N]$. 
On the other hand, by \cref{stronglyproperquotient},  the quotient forcing $\mathbb Q_N$ is $\mathscr C^N_{\rm st}$-strongly proper,
 and $\mathscr C^N_{\rm st}$ is stationary in ${\mathcal P}_{\omega_1}(V_\lambda)$.
It follows by \cref{guessingbystronglyproper} that the pair $(V[G_N], V[G])$ has the $\omega_1$-approximation property.
Thus,  $N[G_N]$ remains an $\omega_1$-guessing model in $V[G]$. 
\end{proof}

A similar argument shows the following. 

\begin{lemma}\label{standard-Magidor-guessing} Suppose $\mu > \lambda$ and $N\prec V_\mu$ is a $\kappa$-Magidor model
containing all the relevant parameters. Then $N[G]$ is an $\omega_1$-guessing model in $V[G]$. 
\end{lemma}

\begin{proof}[\nopunct] 
 \end{proof}
  
Now, by \cref{Magidor-stationary} we have the following. 

 \begin{theorem} The principle ${\rm GM}(\omega_2,\omega_1)$ holds in $V[G]$.
 \end{theorem}
 
\begin{proof}[\nopunct] 
 \end{proof}

\begin{theorem} The principle ${\rm FS}(\omega_2,\omega_1)$ holds in $V[G]$. 
\end{theorem}

\begin{proof}
Fix  $X\in H(\omega_3)^{V[G]}$. We have to find a collection $\mathcal G$ of $\omega_1$-guessing models 
containing $X$ such that $\{ M\cap \omega_2: M \in \mathcal G\}$ is an $\omega_1$-closed unbounded subset of $\omega_2$. 
Back in $V$ we can find $\alpha \in E$, and a canonical $\mathbb P^\kappa_\lambda$-name $\dot X$, such that $\dot X[G_\alpha]=X$.
Fix some  $\beta \in E\setminus (\alpha +1)$ with $\cof(\beta) <\kappa$. 
By a standard density argument, we can find a Magidor model $M\in \M_{G_\beta}$ such that $\alpha, \dot X \in M$. 
Suppose $P\in \M_{G_\beta}$ is also a Magidor model and $M\in_\beta P$. Notice that $M\cap V_\beta \subseteq P\cap V_\beta$, 
so $\dot{X}\in P$, and hence $X\in P[G_\alpha]$. By \cref{Magidor-guessing}, $P[G_\alpha]$ is an $\omega_1$-guessing model, for all such $P$.
Now, by \cref{addingclub}, the set $C_\beta(G)$ is club in $\omega_2$,
and hence the family $\mathcal G= \{ P\in \M_{G_\beta}\cap \mathscr U: M \in_\beta P\}$ is as required. 
\end{proof}

Finally, we observe that if $\lambda$ is also supercompact, then  ${\rm GM}^+(\omega_3,\omega_1)$ holds in $V[G]$ as well. 
In fact, we show that for all $\mu >\lambda$ the set of strong $\omega_1$-guessing models is stationary in ${\mathcal P}_{\omega_3}(V_\mu[G])$. 
 
 \begin{lemma}\label{lambda-Magidor-guessing} Suppose $\mu > \lambda$ and $N\prec V_\mu$ is a $\lambda$-Magidor model
 containing all the relevant parameters. Then $N[G]$ is a strong $\omega_1$-guessing  model. 
\end{lemma}

\begin{proof} Since $N$ is a $\lambda$-Magidor model, its transitive collapse $\overline{N}$  equals $V_{\bar\gamma}$, for some $\bar\gamma < \lambda$.
Let $\bar\lambda = N\cap \lambda$.  Note that $\cof(\bar\lambda)\geq \kappa$, and hence 
the transitive collapse $\overline{N[G]}$  of $N[G]$ equals $V_{\bar\gamma}[G_{\bar\lambda}]$.
On the other hand, by \cref{approx-alpha}, the pair $(V[G_{\bar\lambda}],V)$ has the $\omega_1$-approximation property. 
Moreover, by \cref{quotient-alpha-sp},  ${\mathcal P}_{\omega_1} ^{V_{\bar\gamma}[G_{\bar\lambda}]}(V_{\bar\gamma}[G_{\bar\lambda}])$
is cofinal in  ${\mathcal P}_{\omega_1}^{V[G]}(V_{\bar\gamma}[G_{\bar\lambda}])$. 
Therefore, $V_{\bar\gamma}[G_{\bar\lambda}]$ and hence also $N[G]$ remains an $\omega_1$-guessing model in $V[G]$. 
To see that $V_{\bar\gamma}[G_{\bar\lambda}]$   is a strong $\omega_1$-guessing model, 
fix some $\delta \in E$ with $\delta > \bar\gamma$ and ${\rm cof}(\delta)<\kappa$. 
Note that if $M\in \M_G^\delta$ is a Magidor model with $\bar\lambda \in M$ then by \cref{Magidor-guessing} $M[G_{\bar\lambda}]$ is
an $\omega_1$-guessing model. Moreover, if $M\in \M_G^\delta$ is a limit of such Magidor models
then by \cref{continuity-Magidor},
\[
M\cap V_{\bar\gamma} = \bigcup \{ Q\cap V_\delta : Q\in_\delta M \mbox{ and } Q\in \M_G^\delta\}.
\]
Hence if we let $\mathcal G$ be the collection of the models $(M\cap V_{\bar\gamma})[G_{\bar\lambda}]$,
for Magidor models $M\in \M_G^\delta$ with $\bar\lambda \in M$, then $\mathcal G$ is an increasing $\subseteq$-chain of length $\omega_2$ which is continuous
at $\omega_1$-limits and whose union is $V_{\bar\gamma}[G_{\bar\lambda}]$.
Therefore, $V_{\bar\gamma}[G_{\bar\lambda}]$ and hence also $N[G]$ s a strong $\omega_1$-guessing model in $V[G]$, as required. 
\end{proof}

Now, by \cref{Magidor-stationary} and \cref{lambda-Magidor-guessing} we conclude the following.

\begin{theorem} Suppose $\kappa < \lambda$ are supercompact cardinals. Let $G$ be $V$-generic over $\mathbb P^\kappa_\lambda$. 
Then in $V[G]$ the principle ${\rm GM}^+(\omega_3,\omega_1)$ holds. 
\end{theorem}
\begin{proof}[\nopunct] 
 \end{proof}

 \bibliographystyle{plain}

\bibliography{mybib}

\begin{thebibliography}{10}

\bibitem{Abraham1983}
Uri Abraham.
\newblock Aronszajn trees on {$\aleph _{2}$}\ and {$\aleph _{3}$}.
\newblock {\em Ann. Pure Appl. Logic}, 24(3):213--230, 1983.

\bibitem{CK}
Sean Cox and John Krueger.
\newblock Quotients of strongly proper forcings and guessing models.
\newblock {\em The Journal of Symbolic Logic}, 81(1):264--283, 2016.

\bibitem{CK2017}
Sean Cox and John Krueger.
\newblock Indestructible guessing models and the continuum.
\newblock {\em Fund. Math.}, 239(3):221--258, 2017.

\bibitem{Eisworth2010}
Todd Eisworth.
\newblock Successors of singular cardinals.
\newblock In {\em Handbook of set theory. {V}ols. 1, 2, 3}, pages 1229--1350.
  Springer, Dordrecht, 2010.

\bibitem{Fontanella2012}
Laura Fontanella.
\newblock Strong tree properties for two successive cardinals.
\newblock {\em Arch. Math. Logic}, 51(5-6):601--620, 2012.

\bibitem{Friedman2006}
Sy-David Friedman.
\newblock {\em Forcing with finite conditions}, pages 285--295.
\newblock Trends Math. Birkh\"auser, Basel, 2006.

\bibitem{Hamkins2001}
Joel~David Hamkins.
\newblock Gap forcing.
\newblock {\em Israel Journal of Mathematics}, 125(1):237--252, Dec 2001.

\bibitem{Hamkins2003}
Joel~David Hamkins.
\newblock Extensions with the approximation and cover properties have no new
  large cardinals.
\newblock {\em Fundamenta Mathematicae}, 180(3):257--277, 2003.

\bibitem{Konig2007}
Bernhard K\"onig.
\newblock Forcing indestructibility of set-theoretic axioms.
\newblock {\em J. Symbolic Logic}, 72(1):349--360, 2007.

\bibitem{Kru2019}
John Krueger.
\newblock Guessing models imply the singular cardinal hypothesis.
\newblock Preprint, March 2019.

\bibitem{MA1971}
M.~Magidor.
\newblock On the role of supercompact and extendible cardinals in logic.
\newblock {\em Israel Journal of Mathematics}, 10(2):147--157, Jun 1971.

\bibitem{MI2004}
William~J. Mitchell.
\newblock A weak variation of {S}helah's ${I}[\omega_2]$.
\newblock {\em Journal of Symbolic Logic}, 69(1):94--100, 2004.

\bibitem{MI2005}
William~J. Mitchell.
\newblock Adding closed unbounded subsets of $\omega_2$ with finite forcing.
\newblock {\em Notre Dame J. Formal Logic}, 46(3):357--371, 07 2005.

\bibitem{MI2006}
William~J. Mitchell.
\newblock On the {H}amkins approximation property.
\newblock {\em Ann. Pure Appl. Logic}, 144(1-3):126--129, 2006.

\bibitem{MI2009}
William~J. Mitchell.
\newblock ${I}[\omega_2]$ can be the nonstationary ideal on ${{\rm C}}{\rm
  of}(\omega_1)$.
\newblock {\em Transactions of the American Mathematical Society},
  361(2):561--601, 2009.

\bibitem{NE2014}
Itay Neeman.
\newblock Forcing with sequences of models of two types.
\newblock {\em Notre Dame J. Formal Logic}, 55(2):265--298, 2014.

\bibitem{SakVel2015}
Hiroshi Sakai and Boban Veli\v{c}kovi\'c.
\newblock Stationary reflection principles and two cardinal tree properties.
\newblock {\em J. Inst. Math. Jussieu}, 14(1):69--85, 2015.

\bibitem{Shelah1979}
Saharon Shelah.
\newblock On successors of singular cardinals.
\newblock In {\em Logic {C}olloquium '78 ({M}ons, 1978)}, volume~97 of {\em
  Stud. Logic Foundations Math.}, pages 357--380. North-Holland, Amsterdam-New
  York, 1979.

\bibitem{Shelah1991}
Saharon Shelah.
\newblock Reflecting stationary sets and successors of singular cardinals.
\newblock {\em Archive for Mathematical Logic}, 31(1):25--53, Jan 1991.

\bibitem{Shelah1993}
Saharon Shelah.
\newblock Advances in cardinal arithmetic.
\newblock In {\em Finite and infinite combinatorics in sets and logic ({B}anff,
  {AB}, 1991)}, volume 411 of {\em NATO Adv. Sci. Inst. Ser. C Math. Phys.
  Sci.}, pages 355--383. Kluwer Acad. Publ., Dordrecht, 1993.

\bibitem{Strullu2011}
Remi Strullu.
\newblock {$MRP$}, tree properties and square principles.
\newblock {\em J. Symbolic Logic}, 76(4):1441--1452, 2011.

\bibitem{T2016}
Nam Trang.
\newblock \rm{PFA} and guessing models.
\newblock {\em Israel Journal of Mathematics}, 215(2):607--667, Sep 2016.

\bibitem{Velisemiproper}
Boban Veli\v{c}kovi\'{c}.
\newblock Iteration of semiproper forcing revisited.
\newblock {\em Preprint, 2014}, 2014.

\bibitem{Veliproper}
Boban Veli\v{c}kovi\'{c}.
\newblock Notes on proper and semi-proper forcing.
\newblock {\em In preparation, 2015}, 2015.

\bibitem{V2012}
Matteo Viale.
\newblock Guessing models and generalized {L}aver diamond.
\newblock {\em Ann. Pure Appl. Logic}, 163(11):1660--1678, 2012.

\bibitem{VW2011}
Matteo Viale and Christoph Wei\ss.
\newblock On the consistency strength of the proper forcing axiom.
\newblock {\em Adv. Math.}, 228(5):2672--2687, 2011.

\bibitem{W2010}
Christoph Wei\ss.
\newblock {\em Subtle and Ineffable tree properties}.
\newblock PhD thesis, Ludwig Maximilians Universit\"at M\"unchen, January 2010.

\bibitem{W2012}
Christoph Wei\ss.
\newblock The combinatorial essence of supercompactness.
\newblock {\em Ann. Pure Appl. Logic}, 163(11):1710--1717, 2012.

\end{thebibliography}

\end{document}